%% file: main.tex
\documentclass[a4paper,12pt,times]{article}

\usepackage{setspace}
\usepackage{amstext} 
\usepackage[utf8]{inputenc}
\usepackage{amsmath, amsthm, amsfonts, amssymb}
\usepackage{float}
\usepackage{mathrsfs}
\usepackage{booktabs}
\usepackage{hyperref}
\usepackage{enumitem}
\usepackage{tikz-cd}
\usepackage[margin=1in]{geometry}
\usepackage{CJKutf8}
\usepackage{url}
\usepackage{ytableau}
\usepackage{caption}
\usepackage{tikz}
\usepackage{stmaryrd}
\usepackage{mathabx}
\usepackage{makecell}
\makeatletter
\def\makebb#1{\expandafter\def\csname bb#1\endcsname{{\mathbb{#1}}}\ignorespaces}
\def\makerm#1{\expandafter\def\csname rm#1\endcsname{{\rm #1}}\ignorespaces}
\def\makebf#1{\expandafter\def\csname bf#1\endcsname{{\bf #1}}\ignorespaces}
\def\makegr#1{\expandafter\def\csname gr#1\endcsname{{\mathfrak{#1}}}\ignorespaces}
\def\makescr#1{\expandafter\def\csname scr#1\endcsname{{\mathscr{#1}}}\ignorespaces}
\def\makecal#1{\expandafter\def\csname cal#1\endcsname{{\cal #1}}\ignorespaces}
\def\makeudl#1{\expandafter\def\csname udl#1\endcsname{{\underline{#1}}}\ignorespaces}
\def\doLetters#1{%
  #1A #1B #1C #1D #1E #1F #1G #1H #1I #1J #1K #1L #1M
  #1N #1O #1P #1Q #1R #1S #1T #1U #1V #1W #1X #1Y #1Z}
\def\doletters#1{%
  #1a #1b #1c #1d #1e #1f #1g #1h #1i #1j #1k #1l #1m
  #1n #1o #1p #1q #1r #1s #1t #1u #1v #1w #1x #1y #1z}
\doLetters\makebb \doLetters\makecal \doLetters\makerm \doLetters\makebf \doLetters\makescr
\doletters\makerm \doletters\makebf \doLetters\makegr \doletters\makegr \doLetters\makeudl \doletters\makeudl

\newcounter{InSection}
\numberwithin{InSection}{subsection}
\newtheorem{thm}[InSection]{Theorem}
\newtheorem{lemma}[InSection]{Lemma}
\newtheorem{cor}[InSection]{Corollary}

\newtheorem{prop}[InSection]{Proposition}
\theoremstyle{definition}

\theoremstyle{definition}

\theoremstyle{definition}
\newtheorem{rem}[InSection]{Remark}
\newtheorem*{acknowledgements}{Acknowledgements}
\usepackage{titlesec}

\def\makeop#1{
  \expandafter\def
  \csname#1\endcsname{
    \mathop{\rm #1}
    \nolimits
  }
  \ignorespaces
}

\def\makeoplist#1 {
  \def\@@tmpa{#1}
  \def\@@tmpb{***}
  \ifx\@@tmpa\@@tmpb
  \else
    \makeop{#1}
    \expandafter\makeoplist
  \fi
}

\makeoplist 
ad Ad Alb Alt Aut Br can Card card Char Cl Coker coker Cond cond Corr
depth diag D Diff Disc disc Div div dom End Ext Fil Fr Frac Frob Gal GL
Gr gr GSp GSpin H Hom hom hom  Id id Im im Ind ind Int inv Irr irr Isom
Ker ker len length lk Lie Nr Nrd O Ob ob ord PGL Pic Pr pr Proj PSL rank
RD Res Sgn sgn SL SO Sp Sp sp Spec Spf Spin 
St st Stab stab Std std SU Supp supp Swan Sym Tor tor tors 
torsion Tr tr Trd U UT val Vol vol wt ***

\usepackage{enumitem}
\usepackage{comment}
\usepackage[backend=biber]{biblatex}
\DeclareLanguageMapping{english}{english-apa}

\newcommand{\Fq}{\mathbb{F}_q}
\newcommand{\F}{\mathbb{F}}

\newcommand{\B}{\mathcal{B}}
\newcommand{\C}{\mathbb{C}}

\newcommand{\nex}{\mathrm{next}}
\newcommand{\prev}{\mathrm{prev}}

\newcommand{\Z}{\mathbb{Z}}
\usepackage{tikz}
\usetikzlibrary{positioning}

\begin{document}

\title{Edge Zeta Functions and Eigenvalues for Spherical Buildings: A Uniform Hecke Algebra Approach}

\author{SHEN, Jianhao}
\maketitle

\begin{abstract}

For the Tits building $\B(G)$ of a finite group of Lie type $G(\mathbb{F}_q)$, we study the edge zeta function, which enumerates edge–geodesic cycles in the $1$–skeleton. We show that every nonzero edge eigenvalue becomes a power of $q$ after raising to a bounded exponent $k$ depending on the type of $G$. 

The proof is uniform across types using a Hecke algebra approach. This extends previous results for type~$\mathbf{A}$ and for oppositeness graphs to the full edge–geodesic setting and all finite groups of Lie type.

\end{abstract}
\tableofcontents
\input{sec0}
\input{sec1}
\input{sec2}
\input{sec3}

\input{sec4}

\printbibliography[heading=bibintoc,title={Bibliography}]
\newpage
\end{document}

%% file: sec0.tex
\section*{Introduction}
\addcontentsline{toc}{section}{Introduction}
Zeta functions of spherical buildings reveal deep connections between geometry, combinatorics, and representation theory. 
They generalize the Ihara zeta functions of regular graphs~\cite{Ihara1966} and their two–variable extensions for $q$–regular trees~\cite{Hashimoto1989}.
In higher rank, the geometry of a building reflects the combinatorics of its Weyl group; for buildings arising from algebraic groups, it also encodes the underlying symmetries of the group, linking combinatorial and algebraic structure.
The corresponding zeta function records weighted counts of closed geodesics or galleries, unifying spectral, combinatorial, and arithmetic data in a single analytic object.

The 1–dimensional spherical buildings, as well as finite quotients of trees (1–dimensional affine buildings), form regular graphs whose spectral and zeta–function properties are completely understood through the classical theory of Ihara zeta functions and their variants~\cite{serre2002,feit1964,lubotzky1994}. Motivated by this success, subsequent work extended the theory to higher–dimensional analogues arising from \emph{affine buildings} of reductive groups over local fields.
In particular, the edge and gallery zeta functions of finite quotients of two–dimensional affine buildings have been extensively investigated, revealing deep connections with $p$–adic $L$–functions and automorphic forms (\cite{kang2010,li2011,FLW2013,kangli2014,li2019}).
By contrast, the zeta functions of \emph{spherical buildings} over finite fields remain much less explored.
Beyond the one–dimensional case, previous studies have focused mainly on oppositeness graphs and their spectra~\cite{brouwer2010,sin2012}, leaving the general geodesic structure largely unexplored.

In earlier work~\cite{ShenAIM2024}, the present author introduced the \emph{edge zeta function} for spherical buildings of type~$\mathbf{A}_n$ and derived closed formulas via a character–theoretic analysis of induced Weyl–group representations.
This resolved the type~$\mathbf{A}$ case completely at the level of the 1–skeleton, and called for a generalization to other classical and exceptional types.

The present paper develops such a uniform framework for all finite spherical buildings associated with finite groups of Lie type, establishing the representation–theoretic origin of the corresponding edge eigenvalues.
For a finite thick spherical building~$\B(G)$ associated with a finite group of Lie type~$G(\F_q)$, we study its \emph{edge zeta function}~$Z_{\B}(u)$, which enumerates primitive closed geodesics in the $1$–skeleton of~$\B(G)$.

Our main theorem \ref{Mainthm} asserts that every nonzero edge eigenvalue $\lambda$ becomes a pure power of $q$ after being raised to an integer exponent $k=2m$ depending on the Lie type and the type orbit.  
This establishes a representation–theoretic explanation of the “power–of–$q$” phenomenon first observed in type~$\mathbf{A}$, and shows that the same structure persists across all finite types.  
The proof combines Luo’s recent decomposition theorem~\cite{luo2022} with Springer’s theorem on central elements of Iwahori–Hecke algebras~\cite[\S9.2]{geck_pfeiffer2000}. A refinement of the main result (Theorem~\ref{Mainthmv2}) further gives an explicit description of the associated spectral data and the zeta functions. Finally, we derive new closed formulas for the symplectic case (type~$\mathbf{C}$), and tabulate the corresponding cycle structures and the exponents $2m$.

\paragraph{Organization of the paper}
\begin{itemize}
  \item Section~\ref{sec:1} fixes notation and recalls the definition of geodesic cycles in the 1–skeleton of a spherical building. It then states the main theorem on the algebraic form of the edge eigenvalues.
  
  \item Section~\ref{sec:2} reformulates the problem on the geodesic edge graph $X_2(\B)$ and introduces its partite decomposition and the relative destination elements (RDEs) that encode adjacency.  

  \item Section~\ref{sec:3} performs the Hecke algebra reduction: by Luo’s decomposition and Springer’s centrality theorem, the RDE products collapse to a pair of double–coset operators, with explicit eigenvalue formulas, completing the proof of the main result.  

  \item Section~\ref{sec:4} presents the explicit computations—recovering type~$\mathbf{A}$ formulas, deriving new formulas for type~$\mathbf{C}$, and listing cycle data and half–lengths $m$ for all finite types other than type~$\mathbf{D}$.
\end{itemize}

In this way, the paper extends the type~$\mathbf{A}$ results of~\cite{ShenAIM2024} to all finite spherical buildings arising from groups of Lie type, showing that the edge zeta functions are uniformly controlled by the representation theory of Hecke algebras across all finite Lie types.

%% file: sec1.tex
\section{Settings and Main Theorem}\label{sec:1}

The central object in our analysis is the edge zeta function of the $1$-skeleton of a spherical building, whose coefficients enumerate geodesic cycles.  
This section fixes the notation for the group-theoretic setting, formalizes the notion of geodesics in this context, and states our main theorem on the associated eigenvalues.

\paragraph{Root system and Weyl group of $G$}
Let $G$ be a finite group of Lie type defined over the finite field $\Fq$.  
Then $G$ admits a $BN$-pair, and the Weyl group $W = N / (N \cap B)$ is a Euclidean reflection group of a real vector space $V$ containing a root system $\Phi$.  
We choose a set $\Phi^+ \subset \Phi$ of positive roots, let $\Delta$ be the corresponding set of simple roots, and set
\[
S = \{\, s_\alpha : \alpha \in \Delta \,\} \subseteq W
\]
to be the reflections with respect to roots in $\Delta$.  
Then $(W,S)$ is a Coxeter system.

For $I \subseteq S$, let $\Delta_I$ be the subset of simple roots corresponding to $I$ and $W_I = \langle I\rangle$ the corresponding standard parabolic subgroup of $W$, with longest element $w_I$.  
The set $\Phi_I$ denotes the roots spanned by $\Delta_I$.  
The standard parabolic subgroup of $G$ of type $I$ is $P_I := B W_I B$.

If $\alpha,\beta \in \Delta_I$ satisfy $\beta=-w_I\alpha$, we call them \emph{opposite simple roots} in $\Delta_I$; if $i,j \in I$ satisfy $s_i = w_I s_j w_I$, we call $s_i$ and $s_j$ \emph{opposite simple reflections} in $W_I$. These notions are equivalent: simple roots are opposite exactly when their corresponding
reflections are opposite.  

Note that two opposite roots $\alpha$ and $\beta$ lie in the same connected component of the Dynkin diagram of $\Delta_I$, and $\alpha = \beta$ except when the component is of type $\mathbf{A}$, $\mathbf{D}_n$ ($n$ odd), or $\mathbf{E}_6$.

\paragraph{Spherical Tits buildings}
Let $\B(G)$ be the spherical Tits building associated with $G$, whose simplices are cosets of standard parabolic subgroups of $G$, ordered by reverse inclusion.  
In particular:
\begin{itemize}
    \item Vertices correspond to cosets of maximal parabolic subgroups, $g P_{S-\{r\}}$ with $g \in G$, $r \in S$. We call $r$ the \emph{type} of such a vertex.
    \item Edges correspond to cosets $g P_{S-\{r,s\}}$ for distinct $r,s \in S$.
    \item Two vertices $gP_{S-\{r\}}$ and $hP_{S-\{s\}}$ are \emph{opposite} if their 
      types $r,s$ are opposite in $W$, i.e.\ $s=w_0 r w_0$, and if moreover 
      $g^{-1}h \in P_{S-\{r\}}\, w_0\, P_{S-\{s\}}$, where $w_0$ denotes the longest 
      element of $W$.
\end{itemize}

\paragraph{Geodesics in the $1$-skeleton, zeta functions}
Let $X_1(\B)$ denote the $1$-skeleton of $\B(G)$, which is the graph that consists of vertices and edges of $\B(G)$ and forgets the higher structure. 

Let \(\big(aP_{S-\{r\}},\; bP_{S-\{s\}},\; cP_{S-\{t\}}\big)
\) be a length-$2$ path
in $X_1(\B)$. It is called a \emph{geodesic path} if and only if the first and third vertices are opposite in the link $\lk_{\B}(bP_{S-\{s\}})$.  
Equivalently (See Lemma ~\ref{lem:local-geod-equiv}):
\begin{equation}\label{eq:geodesic-condition-sec1}
    \begin{cases}
        r \text{ and } t \text{ are opposite types in } W_{S-\{s\}}, \\
        a^{-1}c \in P_{S-\{r\}}\, w_{S-\{s\}}\, P_{S-\{t\}}.
    \end{cases}
\end{equation}

A \emph{geodesic cycle} of length $n$ in $X_1(\B)$ is a cyclic sequence $(x_0, \dots, x_{n-1}, x_n=x_0)$ of vertices such that for each $0\leq i\leq n-1$, the triple $(x_{i-1},x_i,x_{i+1})$ is a geodesic path (indices mod~$n$).  
Let $N(l)$ be the number of geodesic cycles of length $l$.  
The \emph{edge zeta function} of $\B(G)$ is
\[
Z(\B, u) \;=\; \exp\!\left( \sum_{l=1}^\infty \frac{N(l)}{l} u^l \right).
\]
Then $1 / Z(\B, u) \in \Z[u]$ is a polynomial with constant term $1$ (\S~\ref{sec:def-X2}); writing
\[
\frac{1}{Z(\B, u)} = \prod_i (1 - \lambda_i u),
\]
the nonzero $\lambda_i$ are called the \emph{edge eigenvalues} of $\B(G)$.

The main theorem of this paper describes the algebraic form of these eigenvalues.
\begin{thm}\label{Mainthm}
Let $G$ be a finite group of Lie type over $\Fq$ and $\B(G)$ its associated building.  
For each edge eigenvalue $\lambda$ of $G$, there exists an integer $k > 0$ such that $\lambda^k$ is an integer power of $q$.
\end{thm}

For type $\mathbf{A}$, one may take $k=6$ for all $\lambda$ (\cite[\S7.2, 8.4]{ShenAIM2024}).  
For types $\mathbf{B}$ or $\mathbf{C}$, one may take $k=8$ (Section~\ref{subsec:typeC}), and for type $\mathbf{D}$, $k$ can be $6$ or $8$ depending on $\lambda$.  
Explicit $k$ (or $m = k/2$) values for various graphs appear in the Section~\ref{sec:4.4}.

\paragraph{Notation and conventions.}

All groups are finite, and all representations are over~$\C$.
For elements \(a,x\in \C G\), write  
\(a_\ell(x)=ax\) and \(a_r(x)=xa\)  
for left and right multiplication, respectively.  
Conjugation on the right by~\(g\in G\) is denoted  
\(a^g=g^{-1}ag\) and \(I^g=g^{-1}Ig\).

For a Weyl group \(W\) with generating set \(S\),  
and \(I\subseteq S\), set  
\(W_I=\langle I\rangle\subseteq W\) and \(w_{W_I}=w_I\)  
for the corresponding standard parabolic subgroup and its longest element.

For any linear operator \(T\), write  
\(\Spec(T)\) for the multiset of eigenvalues of $T$.  
For a finite directed graph \(X\), \(\Spec(X)\) denotes the multiset of eigenvalues of its adjacency operator.

Throughout, \(e_P\) denotes the idempotent 
\(|P|^{-1}\sum_{x\in P}x\) in the group algebra~\(\C G\),  
and \(a_{P w Q}\) denotes the double–coset sum  
\(a_{P w Q}=|Q|^{-1}\sum_{x\in P w Q}x\).  
All tensor products are taken over~\(\C\).

%% file: sec2.tex
\section{From the Building Zeta Function to a Group Algebra Problem}\label{sec:2}
In Section~\ref{sec:1} we defined the building zeta function via closed galleries. We now
reformulate it using the geodesic edge graph $X_2(\B)$, whose closed walks are the geodesics of $\B$ and hence give the same zeta function. We then decompose \(X_2(\mathcal B)\) into type-orbits (multipartite components) and express the component adjacency via relative destination elements (RDEs) in the group algebra, hence reducing the computation of the zeta function to spectral data of explicit Hecke operators.

\subsection{From geodesic cycles to the geodesic edge graph}\label{sec:def-X2}

We begin by verifying a characterization of geodesic paths from Section~\ref{sec:1}.

\begin{lemma}\label{lem:local-geod-equiv}\label{lem:geodesic-criterion}
Let 
\[
x_{-}=aP_{S-\{r\}},\qquad x=bP_{S-\{s\}},\qquad x_{+}=cP_{S-\{t\}}
\]
be a length-$2$ path in $X_1(\B)$. Then the following are equivalent:
\begin{enumerate}
\item[(i)] $x_{-}\!\to x \to x_{+}$ is geodesic at $x$, i.e.\ $x_{-}$ and $x_{+}$ are opposite in the link $\lk_{\B}(x)$.
\item[(ii)] $r$ and $t$ are opposite types in $W_{S-\{s\}}$ and 
\[
a^{-1}c \in P_{S-\{r\}}\; w_{S-\{s\}}\; P_{S-\{t\}}.
\]
\end{enumerate}
\end{lemma}

\begin{proof}
\textbf{Step 1: Normalization.}
Left multiplication by $b^{-1}$ preserves adjacency and opposition, so we may assume $b=1$ and
\[
x = P_{S-\{s\}},\quad x_- = aP_{S-\{r\}},\quad x_+ = cP_{S-\{t\}}.
\]
From adjacency, 
\(
x_- \cap x = aP_{S-\{r\}} \cap P_{S-\{s\}}
\)
is a parabolic subgroup of type $S-\{r,s\}$, hence of the form $a'P_{S-\{r,s\}}$
with $a'\in P_{S-\{s\}}$, and with $a'P_{S-\{r\}}=aP_{S-\{r\}}$, so we may replace
$a$ by $a'$. The same argument applies to $c$. Thus we may assume $a,c\in
P_{S-\{s\}}$ and
\[
(x_-,x)=aP_{S-\{r,s\}},\qquad (x,x_+)=cP_{S-\{s,t\}}.
\]

\textbf{Step 2: Link and the Levi factor.}
Let $U_s$ be the unipotent radical of $P_{S-\{s\}}$ and $L_s:=P_{S-\{s\}}/U_s$ its Levi factor. The link of $x$ is canonically the building of $L_s$:
\[
\lk_\B(x)\;\cong\;\Delta(L_s,\,S-\{s\}).
\]
For vertices of type $r$, the identification is given by modding out $U_s$:
\[
\phi_x^{(r)}:\ 
P_{S-\{s\}}/P_{S-\{r,s\}}
\;\xrightarrow{\ \ \cong\ \ }\;
\frac{P_{S-\{s\}}/U_s}{\,P_{S-\{r,s\}}/U_s\,}=L_s / P^{\,L_s}_{(S-\{s\})-\{r\}},
\]
\[
a'P_{S-\{r,s\}}\ \longmapsto\ \overline{a'}\,P^{\,L_s}_{(S-\{s\})-\{r\}},
\]
where $U_s \subset P_{S-\{r,s\}}$ ensures the map is well-defined, and $P^{\,L_s}_{(S-\{s\})-\{r\}}$ denotes the standard parabolic subgroup of $L_s$ of type $(S-\{s\})-\{r\}$.

Consequently, the two vertices $x_-$ and $x_+$ on $\lk_\B(x)$ project respectively to
\[
v_- \;=\; \phi_x^{(r)}\!\left(aP_{S-\{r,s\}}\right) 
        = \overline{a}\,P^{\,L_s}_{(S-\{s\})-\{r\}},
\qquad
v_+ \;=\; \phi_x^{(t)}\!\left(cP_{S-\{s,t\}}\right) 
        = \overline{c}\,P^{\,L_s}_{(S-\{s\})-\{t\}},
\]
of $\Delta(L_s,S-\{s\})$. 

Hence, condition (i) is equivalent to $v_-$ and $v_+$ being opposite in the link building $\Delta(L_s,S-\{s\})$.  
In building terms, this means that  
$r$ and $t$ are opposite types in $W_{S-\{s\}}$ and
\[
\overline{a}^{-1}\,\overline{c} \ \in\ 
P^{\,L_s}_{(S-\{s\})-\{r\}}\,
w_{S-\{s\}}\,
P^{\,L_s}_{(S-\{s\})-\{t\}}.
\]
Lifting to $G$ gives
\(a^{-1}c \ \in\ P_{S-\{r\}}\,w_{S-\{s\}}\,P_{S-\{t\}}.\)
Therefore, (i) and (ii) are equivalent.
\end{proof}

This local criterion lets us reformulate the problem of counting geodesic cycles in $X_1(\B)$ as counting closed walks in a new directed graph $X_2(\B)$.

\paragraph{Definition of the geodesic edge graph $X_2(\B)$}\label{def:X2}
The \emph{geodesic edge graph} $X_2(\B)$ is defined as follows:
\begin{itemize}
    \item \textbf{Vertices:} a vertex of $X_2(\B)$ is a directed edge of $X_1(\B)$, written as
    \[
    (gP_{S-\{r,s\}},\; r\to s) := \big(gP_{S-\{r\}},\ gP_{S-\{s\}}\big),
    \]
    with $g \in G$ and distinct $r,s \in S$. Its \emph{type} is the ordered pair $(r,s)$.
    \item \textbf{Edges:} there is a directed edge from $(gP_{S-\{r,s\}},\, r \to s)$ to $(hP_{S-\{s,t\}},\, s \to t)$ if and only if $gP_{S-\{s\}} = hP_{S-\{s\}}$ and the condition of Lemma~\ref{lem:local-geod-equiv} holds for the triple
    \[
    \big(gP_{S-\{r\}},\; gP_{S-\{s\}} = hP_{S-\{s\}},\; hP_{S-\{t\}}\big).
    \]
\end{itemize}

By construction, each closed walk in $X_2(\B)$ corresponds uniquely to a geodesic cycle of the same length in $X_1(\B)$, and conversely.  
More precisely:
\[
(x_0\to x_1 \to \dots \to x_{n-1} \to x_0)
\quad\longleftrightarrow\quad
\big((x_0,x_1),\ (x_1,x_2),\ \dots,\ (x_{n-1},x_0),\ (x_0,x_1)\big),
\]
is a bijection between geodesic $n$-cycles in $X_1(\B)$ and closed $n$-walks in $X_2(\B)$.

For a finite directed graph $X$, the \emph{closed-walk zeta function} $Z_c(X,u)$ is defined by
\[
Z_c(X,u) \;=\; \exp\!\left( \sum_{l=1}^\infty \frac{N_c(l)}{l} u^l \right),
\]
where $N_c(l)$ denotes the number of closed walks of length $l$ in $X$.  

Under the above bijection we have $N(l) = N_c(l)$, and hence
\[
Z(X_1(\B),u) \;=\; Z_c(X_2(\B),u),
\]
so the nonzero eigenvalues of $X_1(\B)$ coincide with the nonzero eigenvalues of $X_2(\B)$. Therefore, the edge eigenvalues of $\B(G)$ coincides with the nonzero spectrum of $X_2(\B)$.

This reduction will serve as the starting point for the structural analysis of $X_2(\B)$ in the following subsections.

\emph{Remark.}  
$X_2(\B)$ was  introduced in \cite[\S2]{ShenAIM2024} under the name \emph{geodesic edge graph}.  
From a graph-theoretic perspective, $X_2(\B)$ is a subgraph of the \emph{oriented line graph} of $X_1(\B)$, whose vertices represent directed edges of $X_1(\B)$ and where adjacency encodes consecutive edge-pairs (see \cite{KotaniSunada2000}).   In our case, the adjacency relation is further restricted by the local geodesic condition of Lemma~\ref{lem:local-geod-equiv}.

\subsection{Partite Decomposition of $X_2(\B)$}\label{sec:partite-decomp}
Vertex types in $X_2(\B)$ rotate cyclically under the local geodesic rule.  
This induces a canonical orbit decomposition, splitting $X_2(\B)$ into disjoint multipartite components whose spectra can be studied separately.

\paragraph{Types and adjacency} Vertices of $X_2(\B)$ naturally come with a \emph{type}: a vertex 
\((gP_{S-\{r,s\}},\, r\to s)\) has type $(r,s)$ with $r\neq s\in S$.
By Lemma~\ref{lem:local-geod-equiv}, a vertex of type $(r,s)$ connects only to vertices of type $(s,t)$ with
\[
t \;=\; w_{S-\{s\}}\, r\, w_{S-\{s\}}.
\] We may therefore introduce the next-type map.

\paragraph{The next–type map and its orbits}

For $(r,s)\in S^2 \setminus \Delta(S)$, define
\[
\nex(r,s) := \big(s,\, w_{S-\{s\}} r w_{S-\{s\}}\big),
\]
This map is invertible, with inverse
\[
\prev(r,s) := \big(w_{S-\{r\}} s w_{S-\{r\}},\, r\big).
\]
Hence, $\nex$ partitions $S^2\setminus \Delta(S)$ into disjoint orbits
\[
S^2\setminus \Delta(S) \;=\; C_1 \sqcup \cdots \sqcup C_\ell,
\]
called the \emph{type orbits}.
Each orbit $C_i$ consists of types reachable by iterating $\nex$, and no edge of $X_2(\B)$ connects vertices of different orbits. Thus the graph decomposes canonically into components according to type orbits.

\paragraph{Components and factorization of  zeta function}
For each orbit $C_i$, let $X_2(\B)|_{C_i}$ be the induced subgraph on all vertices of types in $C_i$.  
Each $X_2(\B)|_{C_i}$ is a cyclic $c$-partite directed graph ($c=|C_i|$), and $X_2(\B)$ decomposes naturally as
\[
X_2(\B) \;=\; X_2(\B)|_{C_1}\ \sqcup \cdots \sqcup\ X_2(\B)|_{C_\ell}.
\]

Since the components are disjoint, their closed-walk zeta functions multiply:
\[
Z_c(X_2(\B),u) \;=\; \prod_{i=1}^\ell Z_c\big(X_2(\B)|_{C_i},u\big),
\]
and the spectrum of $X_2(\B)$ is the union of the spectra of the $X_2(\B)|_{C_i}$.  
By Section~\ref{sec:def-X2}, the same holds for $Z(X_1(\B),u)$ and for the edge eigenvalues of $\B(G)$.

Thus the proof of Theorem~\ref{Mainthm} reduces to analyzing the eigenvalues of each component $X_2(\B)|_{C_i}$.

\subsection{Transitive Partite Action and RDE Operators}\label{sec:transitive-RDE}

We fix a type orbit \( C = \{ t_0, \dots, t_{c-1} \} \) (length \( c \)) with
\[
t_i = (r_i, r_{i+1}), \qquad \nex(t_i) = t_{i+1} \quad (\text{indices mod } c),
\]
and analyze the zeta function and eigenvalues of the component \( X_2(\B)|_C \).

\paragraph{Cyclic partite structure and \( G \)-action}
Edges in \( X_2(\B)|_C \) only run from type \( t_i \) to type \( t_{i+1} \), so \( X_2(\B)|_C \) is a cyclic \( c \)-partite directed graph.  
Vertices of type \( t_i = (r_i, r_{i+1}) \) take the form
\[
\left(gP_{S-\{r_i, r_{i+1}\}},\, r_i \to r_{i+1} \right) \qquad (g \in G).
\]
Left multiplication by \( G \) acts transitively on each type class and preserves adjacency (as per Lemma~\ref{lem:local-geod-equiv} (ii)).  
Thus, \( G \) acts \emph{partite-transitively} on \( X_2(\B)|_C \) in terms of \cite[\S~3.2]{ShenAIM2024}.

\paragraph{Eigenvalues for partite-transitive graphs (review)}
We now recall the group algebra method to find zeta functions for partite-transitive graphs(cf. \cite[\S2–3]{ShenAIM2024}):

Let $X=(V,E)$ be a directed graph such that $V=\bigsqcup_{i=0}^{c-1} V_i$ and $E\subseteq \bigcup_{i=0}^{c-1} V_i\times V_{i+1}$ (indices mod c). Suppose a finite group $G$ acts on $X$, and $G$ preserves each $V_i$ and acts transitively on each $V_i$. 
\begin{enumerate}
\item[(1)] \textbf{Basepoints and stabilizers.}  
Pick \( v_i \in V_i \) (the set of vertices of type \( t_i \)) and let
\(
P_i := \Stab_G(v_i).
\)

\item[(2)] \textbf{Block adjacency maps.}  
Let \( \C[V_i] \) denote the complex span of \( V_i \), and define the linear map
\[
T_i: \C[V_i] \to \C[V_{i+1}], \qquad v \mapsto \sum_{v \to w} w.
\]
Set \( T_C := T_{c-1} \cdots T_1 T_0: \C[V_0] \to \C[V_0] \). Then, the zeta function
\[
Z_c \left( X_2(\B)|_C, u \right) = \prod_{\lambda\in\Spec(T_C)} \frac{1}{1 - \lambda u^c},
\]
where \( \lambda \) runs over the eigenvalues of \( T_C \) (See for example \cite{ShenAIM2024} Prop. 2.1.11 and Prop. 3.2.9).  Hence, the nonzero eigenvalues of \( X_2(\B)|_C \) are the \( c \)-th roots of the eigenvalues of \( T_C \).

\smallskip
\noindent
\textbf{Remark.}
Each map \(T_i\) is a \emph{Hecke operator}: it acts by summing over the $w$ such that there is an edge from $v$ to $w$.
Thus, \(T_C\) represents a composite Hecke operator corresponding to the $c$-step product of such double cosets.

\item[(3)] \textbf{Hecke operators and RDEs.}  
Define the map
\[
\alpha_i: \C[V_i] \xrightarrow{\ \cong \ }\C[G] e_{P_i}, \qquad g \cdot v_i \mapsto g e_{P_i},
\]
where \( e_{P_i} := |P_i|^{-1} \sum_{x \in P_i} x \in \C[G] \).  
Because \( T_i \) is \( G \)-equivariant, there exists a \( G \)-linear map \( \phi_i \) making the following diagram commute. 
\begin{center}
\begin{tikzcd}[column sep=large]
\C[V_i] \arrow[r, "\alpha_i"] \arrow[d, "T_i"'] & \C G e_{P_i} \arrow[d, "{\phi_i = D(v_i, v_{i+1}, 1)}_r"] \\
\C[V_{i+1}] \arrow[r, "\alpha_{i+1}"'] & \C G e_{P_{i+1}}
\end{tikzcd}
\end{center}

Since \( \phi_i \) is \( G \)-linear, it is given by right multiplication by an element
\[
D(v_i, v_{i+1}, 1) \in e_{P_i} \C[G] e_{P_{i+1}},
\]
called a 1-step \emph{relative destination element} (RDE). Here, $x_r$ denotes right multiplication by $x$.

\item[(4)] \textbf{\( c \)-step operator as a product of RDEs.}  
Under \( \alpha_0 \), the \( c \)-step map \( T_C \) becomes 
\[
\begin{aligned}
T_C &= T_{c-1} \circ \cdots \circ T_1 \circ T_0 \\
    &= D(v_{c-1}, v_0, 1)_r \circ \cdots \circ D(v_1, v_2, 1)_r \circ D(v_0, v_1, 1)_r \\
    &= (D(v_0, v_1, 1) D(v_1, v_2, 1) \cdots D(v_{c-1}, v_0, 1))_r
\end{aligned}
\]
Let \( D(v_0, v_0, c) = D(v_0, v_1, 1) D(v_1, v_2, 1) \dots D(v_{c-1}, v_0, 1) \). 
Thus, the eigenvalues of \( T_C \) coincide with those of \( D(v_0, v_0, c)_r \) on \( \C[G] e_{P_0} \). We call $D(v_0, v_0, c)$ a $c$-step relative destination element
\end{enumerate}

\paragraph{Applying the procedure}
We may now apply the procedure to the component \( X_2(\B)|_C\).
\begin{enumerate}
\item[(1)] \textbf{Base points.}   For the component \( X_2(\B)|_C \), fix basepoints
\[v_i = (P_{S-\{r_i,r_{i+1}\}},\, r_i \to r_{i+1})\] for each type $r_i \to r_{i+1}$.
Then \( \Stab_G(v_i) = P_i = P_{S-\{r_i,r_{i+1}\}} \), and \( v_{i+1} = (P_{i+1}, r_{i+1} \to r_{i+2}) \).  

\item[(2)] \textbf{Length–1 RDE.}  By the commutative diagram above,  
\[
D(v_i, v_{i+1}, 1) = \phi_i(e_{P_i}) = \phi_i(\alpha_i(v_i)) = \alpha_{v_{i+1}}(T_i(v_i)) = \alpha_{v_{i+1}} \left( \sum_{v_i \to w} w \right).
\]

For \( g \in G \), there is an edge from \( v_i \) to \( g v_{i+1} \) if and only if \(  g P_{i+1} \subseteq P_i w_{S-\{r_{i+1}\}} P_{i+1} \), by lemma \ref{lem:geodesic-criterion}.  
Write the maximal length element  \( w_{S-\{r_{i+1}\}} \) in $W_{S-\{r_{i+1}\}}$ as \( w_{i+1} \). Then we have
\[
\sum_{v_i \to w} w = \sum_{gP_{i+1} \subseteq P_i w_{i+1} P_{i+1}} g v_{i+1}.
\]
Thus,
\[
D(v_i, v_{i+1}, 1) = \alpha_{v_{i+1}} \left( \sum_{v_i \to w} w \right) = \sum_{gP_{i+1} \subseteq P_i w_{i+1} P_{i+1}} g e_{P_{i+1}}=\frac{1}{|P_{i+1}|}\sum_{x\in P_i w_{i+1}P_{i+1}}x = a_{P_i w_{i+1} P_{i+1}}.
\]

\item[(3)] \textbf{Zeta function.}  
Consequently, the \( c \)-step operator is
\[
D(v_0,v_0,c)   = D(v_0,v_1,1)\,D(v_1,v_2,1)\cdots D(v_{c-1},v_0,1)   = a_{P_0 w_1 P_1}\,a_{P_1 w_2 P_2}\cdots a_{P_{c-1} w_0 P_0}.
\]
and the nonzero eigenvalues of \( X_2(\B)|_C \) are exactly the \( c \)-th roots of
the eigenvalues of \( D(v_0,v_0,c)_r \) acting on \( \C[G] e_{P_0} \).

The
associated zeta function is therefore
\[
Z_c\!\left(X_2(\B)|_C,u\right) 
   = \prod_{\lambda}
     \frac{1}{1-\lambda u^c},
\]
where \(\lambda\) ranges over eigenvalues of \( D(v_0,v_0,c)_r \) on \( \C[G]e_{P_0} \).
\end{enumerate}
\medskip
In the next subsection, we simplify \( D(v_0, v_0, c) \) using Weyl group length-additivity and Luo’s theorem on reduced products, collapsing the above product to a small number of double-coset operators.  
This, together with Springer’s theorem (as in \cite{brouwer2010, ShenAIM2024}), yields that the resulting eigenvalues are powers of \( q \), completing the proof of Theorem~\ref{Mainthm} componentwise.

%% file: sec3.tex
\section{Hecke Algebra Reduction and Proof of the Main Theorem}
\label{sec:3}
In this section we complete the proof of the main theorem by reducing the
products of the relative destination elements (RDEs) in the Hecke algebra.
Using Luo’s decomposition of the longest Weyl element, we show that after a
finite number of steps the long product of edge operators collapses to a pair
of double–coset operators between opposite parabolic subgroups.
Their spectra are then determined via Springer’s theorem on central elements
of Hecke algebras, leading to explicit expressions for the eigenvalues and the
associated zeta factors.

\subsection{Collapsing the RDE product}\label{sec:computing-RDE}

Fix a type orbit \( C = \{t_0, \dots, t_{c-1}\} \) in the geodesic edge graph \( X_2(\mathcal{B}) \), where 
\[
t_i = (r_i, r_{i+1}), \qquad \nex(t_i) = t_{i+1} \quad (\text{indices mod } c),
\] 
see \S \ref{sec:partite-decomp}. Choose basepoints
\[
v_i = \bigl(P_{S-\{r_i, r_{i+1}\}},\, r_i \to r_{i+1}\bigr), 
\qquad 
P_i := \Stab_G(v_i) = P_{S-\{r_i, r_{i+1}\}}.
\]

We now compute the product
\[
D(v_0, v_0, c) := D(v_0, v_1, 1) \cdots D(v_{c-1}, v_0, 1),
\]
where \( D(v_i, v_{i+1}, 1) = a_{P_i w_{i+1} P_{i+1}} \in e_{P_i} \mathbb{C}G e_{P_{i+1}} \), as discussed in \S \ref{sec:transitive-RDE}. Here \( w_i = w_{S-\{r_i\}} \) is the longest element in \( W_{S-\{r_i\}} \).

The main result is that a suitable power of $D(v_0,v_0,c)$ reduces to a simple form
whose eigenvalues are computed in the next subsection.

\begin{prop}\label{Prop:Collapsing RDE}
There exists a positive integer \( m \) such that \( c \mid 2m \), and
\[
D(v_0, v_0, c)^{2m/c} = a_{P_0 w_S P_m}\, a_{P_0 w_S P_m},
\]
where \( P_m \) and \( P_0 \) are opposite standard parabolic subgroups.
\end{prop}

This proposition is a drastic simplification of the product of $2m$ individual RDEs. We prove this proposition at the end of this subsection, after stating the theorem and lemmas needed.

\paragraph{Luo’s decomposition and its consequences}
\begin{thm}[Luo's Decomposition]\label{thm:Luo-decomp}
Let \( (W, S) \) be the Weyl group associated with a reduced root system \( \Phi \). Let \( s_0, s_1 \in S \) be distinct and define 
\[
s_{n+1} = w_{S-\{s_n\}}\, s_{n-1}\, w_{S-\{s_n\}}, \qquad n \geq 1.
\]
Let \( w'_i \) be the longest element in \( W_{S-\{i,i+1\}} \), and \( w_i \) be the longest element in \( W_{S-\{i\}} \). Then there exists a unique integer \( m \) such that
\(
w'_0 w_S = w'_0 w_1 w'_1 w_2 \cdots w'_{m-1} w_m,
\)
and
\(
l(w'_0 w_S) = \sum_{k=1}^m l(w'_{k-1} w_k).
\)
\end{thm}
We will prove this theorem again in Section~\ref{sec:Luo-new}. We state two corollaries of this decomposition theorem.

\begin{cor}[Invariance of the number of factors]\label{cor:invariance-of-m}
Let $\{s_i\}_{i\in\Bbb Z}$ be the sequence from Luo’s construction, and set $w_i:=w_{S-\{s_i\}}$ and $w'_i:=w_{S-\{s_i,s_{i+1}\}}$.
For each starting index $k\ge0$, suppose Luo’s factorization holds with length $m_k$:
\begin{equation}\label{eq:Luo-k}
  (w'_k w_{k+1})(w'_{k+1}w_{k+2})\cdots (w'_{k+m_k-1}w_{k+m_k}) \;=\; w'_k\,w_S,
  \quad
  \sum_{j=1}^{m_k}\ell\!\bigl(w'_{k+j-1}w_{k+j}\bigr)=\ell(w'_k w_S).
\end{equation}
Then all $m_k$ are equal; i.e. $m_k\equiv m$ is independent of $k$.
\end{cor}

\begin{proof}
By equation \eqref{eq:Luo-k}, $\sum_{j=1}^{m_k}\ell\!\bigl(w'_{k+j-1}w_{k+j}\bigr)=\ell(w'_k w_S).$

Since $w_{i+1}$ is the longest element in $W_{S-\{s_{i+1}\}} \supset W_{S-\{s_i,s_{i+1}\}}$, 
we have
\[
   \ell(w'_i w_{i+1}) = \ell(w_{i+1}) - \ell(w'_i).
\]
Note also that $\ell(w'_k \, w_S)=\ell(w_S)-\ell(w'_k)$. Hence,
\begin{equation}\label{eq:lengthA}
   \sum_{i=k+1}^{k+m_k}\ell(w_i)\;-\;\sum_{i=k+1}^{k+m_k-1}\ell(w'_i)
   \;=\;\ell(w_S).
\end{equation}

Now apply the same reasoning with $k\mapsto k+1$:
\begin{equation}\label{eq:lengthB}
   \sum_{i=k+2}^{k+1+m_{k+1}}\ell(w_i)\;-\;\sum_{i=k+2}^{k+m_{k+1}}\ell(w'_i)
   \;=\;\ell(w_S).
\end{equation}

If $m_{k+1}<m_k$, then subtracting \eqref{eq:lengthB} from \eqref{eq:lengthA} yields
\begin{equation}\label{eq:diff}
   \ell(w_{k+1})-\ell(w'_{k+1})
   \;+\;\sum_{i=k+m_{k+1}+1}^{k+m_k}(\ell(w_i)
   \;-\;\ell(w'_{i-1}))
   \;=\;0.
\end{equation}

Moreover, for every $i$ we have $\ell(w_i)>\ell(w'_i)$ and $\ell(w_{i+1})>\ell(w'_i)$ because 
$W_{S-\{s_i,s_{i+1}\}}\subsetneq W_{S-\{s_i\}}$ and $W_{S-\{s_i,s_{i+1}\}}\subsetneq W_{S-\{s_{i+1}\}}$.  
Thus the left-hand side of \eqref{eq:diff} is strictly positive, a contradiction.

Hence $m_{k+1}\ge m_k$ for all $k$.  
Since the sequence $\{s_i\}$ is periodic with period $c$, the integer sequence $\{m_k\}$ is also $c$-periodic.  
But a weakly increasing, $c$-periodic sequence must be constant,
therefore, all $m_k$ are equal.
\end{proof}

\begin{cor}\label{cor:WS-conjugates-shift}
In the setting of Luo's theorem with the notation above, let $m$ be the common value from Corollary~\ref{cor:invariance-of-m}. Then for every $i$ one has
\[
   s_i^{\,w_S}=s_{i+m}.
\]
\end{cor}

\begin{proof}
From Luo's decomposition one has
\[
   w'_0 w_S=(w'_0w_1)(w'_1w_2)\cdots(w'_{m-1}w_m).
\]
Since each $w_j^2=1$, this gives
\[
   w_S=w_1(w'_1w_2)\cdots(w'_{m-1}w_m), 
   \qquad 
   w_1w_S=(w'_1w_2)\cdots(w'_{m-1}w_m).
\]

Next, by the invariance of $m$, we may shift the indices in Luo’s decomposition. 
Moreover $w'_{j-1}w_j=w_jw'_j$, because $w_j=w_{I-\{s_j\}}$ conjugates 
$I-\{s_j,s_{j-1}\}$ to $I-\{s_j,s_{j+1}\}$ and hence conjugates the corresponding 
longest elements $w'_{j-1}=w_{I-\{s_j,s_{j-1}\}}$ to 
$w'_j=w_{I-\{s_j,s_{j+1}\}}$.

Applying this relation with a shifted index gives
\[
   w_S=w_2(w'_2w_3)\cdots(w'_mw_{m+1})
      =(w'_1w_2)\cdots(w'_{m-1}w_m)w_{m+1}.
\]
Hence$w_1w_S=w_Sw_{m+1}.$

Thus $w_S$ conjugates the longest element of $W_{S\setminus\{s_1\}}$ to that of $W_{S\setminus\{s_{m+1}\}}$.  
Since conjugation by $w_S$ permutes simple reflections, this forces
\[
   w_S(S\setminus\{s_1\})w_S=S\setminus\{s_{m+1}\},
\]
so $s_1^{\,w_S}=s_{m+1}$.  
Shifting indices yields $s_i^{\,w_S}=s_{i+m}$ for all $i$.
\end{proof}

\paragraph{Concatenation lemma for double–coset operators}
We now prove a lemma that allows concatenation of products of elements of the form $a_{P u P'}$.

\begin{lemma}\label{concatenateA}
Let \( u_1, \dots, u_k \in W \), and \( I_0, I_1, \dots, I_k \subseteq S \) such that 
\[
I_{i-1} u_i = u_i I_i.
\]
Let \( W'_i = \langle I_i \rangle \), and \( P_i = B W'_i B \). Suppose that the \( u_i \) are minimal-length elements in \( W'_{i-1} u_i = u_i W'_i \). 
Suppose further that 
\[
l(u_1 \cdots u_k) = l(u_1) + \cdots + l(u_k).
\]
Then \( u_1 \cdots u_k \) is the minimal-length element in \( W_0 u_1 \cdots u_k \), and
\[
a_{P_0 u_1 P_1} \cdots a_{P_{k-1} u_k P_k} = a_{P_0 u_1 \cdots u_k P_k}.
\]
\end{lemma}

\begin{proof}
It suffices to prove the result for \( k=2 \); the general case follows by induction.

The statement that \( u_1 u_2 \) is the minimal-length element in \( W_0 u_1 u_2 \) follows from the deletion condition: 
Suppose \( u_1 u_2 \) is not minimal. Then there exists \( s \in I_0 \) such that 
\[
l(s u_1 u_2) < l(u_1 u_2).
\]
Cancellation cannot occur at \( u_1 \), since \( l(u_1 u_2) = l(u_1) + l(u_2) \) and \( u_1 \) is minimal in \( W_0 u_1 \). 
Thus \( s u_1 u_2 = u_1 u_2' \) for some \( u_2' \) of length \( l(u_2) - 1 \). Then \( s^{u_1} u_2 = u_2' \), but \( s^{u_1} \in I_0^{u_1} = I_1 \), a contradiction. 
Hence \( u_1 u_2 \) is of minimal length.

Now, to prove the product formula: 
\[
a_{P_0 u_1 P_1} a_{P_1 u_2 P_2} = a_{P_0 u_1 u_2 P_2}.
\]

Define a partial order \( \geq \) on \( \mathbb{C} G \) by \( a \geq b \) iff \( a-b = \sum_{g \in G} c_g g \) with \( c_g \geq 0 \) for all \( g \in G \). Then \( \mathbb{C} G_{\geq 0} \) is closed under addition and multiplication. 

First, we claim 
\[
a_{P_0 u_1 P_1} a_{P_1 u_2 P_2} \geq a_{P_0 u_1 u_2 P_2}.
\]
Indeed,
\[
a_{P_0 u_1 P_1} = \frac{1}{|P_1|} \sum_{x \in P_0 u_1 P_1} x \;\geq\; u_1 e_{P_1}.
\]
Thus,
\[
a_{P_0 u_1 P_1} a_{P_1 u_2 P_2} \geq u_1 e_{P_1} a_{P_1 u_2 P_2} = u_1 a_{P_1 u_2 P_2} \geq u_1 u_2 e_{P_2}.
\]
Then, for any \( g \in P_0 \),
\[
a_{P_0 u_1 P_1} a_{P_1 u_2 P_2} = g a_{P_0 u_1 P_1} a_{P_1 u_2 P_2} \geq g u_1 u_2 e_{P_2}.
\]
Hence,
\[
a_{P_0 u_1 u_2 P_2} 
= \sum_{g P_2 \subseteq P_0 u_1 u_2 P_2} g e_{P_2} 
= \sup_{g \in P_0} (g u_1 u_2 e_{P_2})
\;\leq\; a_{P_0 u_1 P_1} a_{P_1 u_2 P_2}.
\]

Secondly, by Lemma \ref{sizeofdoublecoset-unequal}, for some weight function $Q:W\to \mathbb{Z}$,
\[
|P_0 u_1 P_1| = Q(u_1)\,|P_1|, \quad 
|P_1 u_2 P_2| = Q(u_2)\,|P_2|, \quad 
|P_0 u_1 u_2 P_2| = Q(u_1 u_2)\,|P_2|.
\] 

Thus, the sum of coefficients of 
\( a_{P_0 u_1 P_1}, a_{P_1 u_2 P_2}, a_{P_0 u_1 u_2 P_2} \) 
are \( Q(u_1), Q(u_2), Q(u_1 u_2) \) respectively. 
Since $\ell(u_1u_2)=\ell(u_1)+\ell(u_2)$, multiplicativity (see the formula for $Q(w)$ below) gives 
\(Q(u_1u_2)=Q(u_1)Q(u_2)\). 
Therefore, the equality holds.
\end{proof}

Before formulating the size of double cosets in the unequal parameter case, 
we introduce the standard weight function. 
For each simple reflection $s\in S$, let 
\[
Q(s) := \#\{\text{chambers in $B(G)$ containing a fixed panel of type $s$}\},
\]
the \emph{thickness parameter} of type $s$. 
In algebraic terms,
\[Q(s)=\frac{|BsB|}{|B|},
\]
so $Q(s)$ is the Hecke parameter attached to $s$, and may vary across 
different conjugacy classes of simple reflections.  

For $w\in W$ define
\[
Q(w) \;:=\; \prod_{s\in S} Q(s)^{m_s(w)},
\]
where $m_s(w)$ is the number of occurrences of $s$ in any reduced expression of $w$. 
This is well-defined because the parameters $Q(s)$ are constant on conjugacy classes of simple reflections.  
Then
\[
|BwB| \;=\; Q(w)\,|B|,
\]
see e.g.\ \cite[Prop.~3.2]{iwahori-matsumoto1965}.

\begin{lemma}\label{sizeofdoublecoset-unequal}
Let \( I,J \subseteq S \) and \( w \in W \) be such that \( I^w = J \). 
Suppose \( w \) is the minimal-length element in \( W_I w = w W_J \). 
Let \( P_I = BW_I B \), \( P_J = BW_J B \). Then
\[
|P_I| = |P_J|, \qquad |P_I w P_J| = Q(w)\,|P_I|.
\]
\end{lemma}

\begin{proof}
Recall (see Lemma~8.2 in \cite{feit1964}) that 
\[
BW_I B w B = BW_I w B, 
\quad BwB W_I B = Bw W_I B.
\]
Hence
\[
\begin{aligned}
P_I w P_J 
&= BW_I B w B W_J B \\
&= BW_I w W_J B \\
&= BW_I w B.
\end{aligned}
\]
Since $w$ has minimal length in $W_I w$,  we have $\ell(uw)=\ell(u)+\ell(w)$ and $Q(uw)=Q(u)Q(w)$ for any $u\in W_I$. It follows that
\[
\sum_{x \in P_I w P_J} x 
= \sum_{u \in W_I} Q(uw)\,|B| 
= Q(w) \sum_{u \in W_I} Q(u)\,|B| 
= Q(w)\,|P_I|.
\]
Similarly, $P_I w P_J = Bw W_J B$ shows $|P_I w P_J| = Q(w)|P_J|$, so in particular $|P_I|=|P_J|$.
\end{proof}

\paragraph{Proof of Proposition~\ref{Prop:Collapsing RDE}}
We are given \( r_0, r_1 \), with \( r_{i+1} = w_i r_{i-1} w_i \). This defines \( \{r_i\}_{i \in \mathbb{Z}} \), and let \( c \) be the minimal positive period of this sequence. Applying Luo’s theorem \ref{thm:Luo-decomp} with \( s_i = r_i \), we obtain an integer \( m \) such that the decomposition and length condition hold. By Corollary \ref{cor:WS-conjugates-shift}, \( s_k^{w_S} = s_{k+m} \). Thus \( 2m \) is a period for \( \{s_k\} \), hence \( c \mid 2m \).

Next, write
\[
D(v_i, v_j, j-i) = a_{P_i w_{i+1} P_{i+1}} \cdots a_{P_{j-1} w_j P_j}.
\]
Then
\[\begin{aligned}
    D(v_0, v_0, c)^{2m/c} &= D(v_0, v_c, c)\, D(v_c, v_{2c}, c) \cdots D(v_{2m-c}, v_{2m}, c).\\
    &= D(v_0, v_{2m}, 2m) = D(v_0, v_m, m)\, D(v_m, v_{2m}, m).
\end{aligned}
\]

Now, let \( u_i = w'_{i-1} w_i \). Then \( u_i \) is the minimal-length element in
\[
W'_{i-1} w_i = W_{S-\{i-1,i\}} w_{S-\{i\}} = w_{S-\{i\}} W_{S-\{i,i+1\}} = w_i W'_i,
\]
and moreover, \( I_{i-1}=S-\{s_{i-1},s_{i}\} \) is conjugated to \( I_i=S-\{s_{i},s_{i+1}\}\) via $u_i$: $w'_{i-1}$ is the longest element in $W_{I_{i-1}}$ and fixes $I_{i-1}$, while $w_i=w_{S-\{s_i\}}$ conjugates  $I_{i-1}$ to $I_{i}$, and so 
\[
I_{i-1}^{u_i}=(I_{i-1}^{w'_{i-1}})^{w_i} =I_{i-1}^{w_i} =I_i.
\]
Applying Lemma \ref{concatenateA} to the \( u_i \), we obtain
\[
D(v_0, v_m, m) = a_{P_0 u_1 \cdots u_m P_m} =a_{P_0w_0' w_S P_m} = a_{P_0 w_S P_m}.
\]
Similarly,
\[
D(v_m, v_{2m}, m) = a_{P_m w_S P_0}.
\]
Thus
\[
D(v_0, v_0, c)^m = a_{P_0 w_S P_m}\, a_{P_m w_S P_0},
\]
By Corollary \ref{cor:WS-conjugates-shift} again, $I_m^{w_S}=I_0$, and so  \( P_m=BW_{I_m}B\) is opposite to \( P_0=BW_{I_0}B\).
\qed

\subsection{Eigenvalues and edge–zeta factors}\label{sec:3.2}

We now explain that the eigenvalues of the operator
\[
(a_{P_0 w_S P_m}\, a_{P_m w_S P_0})_r : \mathbb{C}G e_{P_0} \;\to\; \mathbb{C}G e_{P_0}
\]
are integer powers of \(q\), where \(w_S\) is the longest element in \(W\), and 
\[
P_0 = B W_I B, \qquad P_m = B W_{I^{w_S}} B.
\]

This fact is closely related to the theory of oppositeness graphs, cf.~\cite{brouwer2010}, and generalizes Proposition~7.0.1 of \cite{ShenAIM2024}. 
Both proofs, ours and the one in \cite{brouwer2010}, are ultimately based on Springer’s theorem on the action of central elements in Iwahori–Hecke algebras. 

\begin{thm}\label{thm:eigenvalues}
Let \( I \subseteq S \), let \( w \) be the longest element in \(W\), and put \( J := I^w \). 
Let \( P_I = B W_I B \) and \( P_J = B W_J B \). 
Then the eigenvalues of right multiplication
\[
(a_{P_I w P_J}\, a_{P_J w P_I})_r : \mathbb{C}G e_{P_I} \to \mathbb{C}G e_{P_I}
\]
are powers of \(q\). 

More precisely: let $\chi$ be an irreducible representation of $G$. 
Let $n_\chi$ be its multiplicity in $\C G e_{P_I}$, and $d_\chi = \dim(\chi)$. 
Define
\[
   f_\chi \;=\; \sum_{t \in \mathcal{R}} \Bigl(1 + \tfrac{\chi_1(t)}{\chi_1(1)}\Bigr),
\]
where $\mathcal{R}$ runs over all conjugacy classes of reflections in the Weyl group $W$, 
and $\chi_1$ denotes the corresponding character of $W$ associated with $\chi$. 
Then the eigenvalues of $(a_{P_I w P_J}\, a_{P_J w P_I})_r$ are $Q(w_I)^{-2}\cdot q^{\,f_\chi}$,
with multiplicity $d_\chi n_\chi$, as $\chi$ ranges over all irreducible representations 
of $G$ with $n_\chi \neq 0$.
\end{thm}

\begin{proof}
First, we reduce to the central element \( a_{B w B} \).  
By Lemma~\ref{sizeofdoublecoset-unequal}, 
\[
P_I w P_J = B W_I w B = P_I w B, \qquad P_J w P_I = B w P_I.
\]
Hence
\[
a_{P_I w P_J}\, a_{P_J w P_I} \;\sim\; e_{P_I}\, a_{B w B}^2 \, e_{P_I},
\]
where \( a \sim b \) means \( a = \lambda b \) for some scalar \( \lambda \neq 0 \). 

By Springer’s Theorem~9.2.2 in \cite{geck_pfeiffer2000}, the element \( a_{B w B}^2 \) lies in the center of the Hecke algebra \( e_B \mathbb{C}G e_B \). 
Since \( e_{P_I} \in e_B \mathbb{C}G e_B \), we have
\[
e_{P_I} a_{B w B}^2 e_{P_I} = e_{P_I} a_{B w B}^2.
\]
Moreover, \((e_{P_I})_r\) acts trivially on \( \mathbb{C}G e_{P_I} \). 
Thus the operator \((a_{P_I w P_J} a_{P_J w P_I})_r\) is equivalent to a scalar multiple of \((a_{B w B}^2)_r\) on \( \mathbb{C}G e_{P_I} \). 

The scalar is determined by total coefficients: 
the minimal element in \( W_I w \) is \(w_Iw\), so by Lemma \ref{sizeofdoublecoset-unequal},
\[
|P_I w P_J| = |P_J w P_I| =Q(w_Iw) |P_J|= \frac{Q(w)}{Q(w_I)} |P_J|=\frac{Q(w)}{Q(w_I)} |P_I|.
\]
Then the total coefficient of $a_{P_I w P_J}=\frac{1}{|P_J|}\sum_{x\in P_IwP_J} x$ is 
$\frac{Q(w)}{Q(w_I)}$, and the same holds for $a_{P_I w P_J}$.  In comparison, $|BwB|=Q(w)|B|$, and $a_{BwB}$ has total coefficient $Q(w)$.
Thus
\[
(a_{P_I w P_J} a_{P_J w P_I})_r = Q(w_I)^{-2} (a_{B w B}^2)_r.
\]

Second, we change sides (Lemma~6.1.1 of \cite{ShenAIM2024}):  
the eigenvalues of \((a_{B w B}^2)_r\) on \( \mathbb{C}G e_{P_I} \) coincide with those of \((a_{B w B}^2)_l\) on \( e_{P_I} \mathbb{C}G \).

Now decompose the left regular representation:
\[
\mathbb{C}G \;\simeq\; \bigoplus_\chi M_\chi^{d_\chi},
\]
where each \( M_\chi \) is an irreducible \(G\)-module. 
Then, as $e_{P_I}\C Ge_{P_I}$-modules,
\[
e_{P_I} \mathbb{C}G \;\simeq\; \bigoplus_\chi (e_{P_I} M_\chi)^{d_\chi}.
\]
By definition, 
\[
\dim(e_{P_I} M_\chi) =\dim \Hom_G(\C Ge_{P_I}, M_{\chi}) =\langle \chi, (1_{P_I})^G \rangle = n_\chi,
\]
the multiplicity of \(M_\chi\) in \(\mathbb{C}G e_{P_I}\). 

Finally, by Springer’s theorem again, the central element \( a_{B w B}^2\in e_B\C Ge_B \) acts on \( e_B M_\chi \), and on $e_P M_{\chi}\subseteq e_BM_{\chi}$, as scalar multiplication by \( q^{f_\chi} \). 
Hence the eigenvalues of \((a_{P_I w P_J} a_{P_J w P_I})_r\) are
\[
Q(w_I)^{-2}  \cdot q^{f_\chi}
\]
with multiplicity \( d_\chi n_\chi \).
\end{proof}

\subsection{Summary: proof and refinement of the main theorem}\label{subsec:main-conclusion}

We now deduce the main theorem \ref{Mainthm} on the spectrum along a type orbit 
and its zeta factor.

Fix a type orbit $C$ of length $c$, with associated $c$-step operator $T_C$ on $\C Ge_{P_0}$.  
By Section~\ref{sec:transitive-RDE}, the zeta factor attached to $C$ is determined by the spectrum of $T_C$, and
\[
T_C \;=\; D(v_0,v_c,c)_r
  \;=\; \bigl(a_{P_0 w_1 P_1}\dots a_{P_{c-1} w_c P_c}\bigr)_r\;.
\]
This is the same as the spectrum of the left multiplication operator $T_C' := D(v_0,v_0,c)_l$ acting on $e_{P_I}\C G$.

We have the following identifications:

\begin{enumerate}
\item By Corollary~\ref{cor:invariance-of-m} and \S\ref{sec:computing-RDE}, there exists \(m>0\) with \(c\mid 2m\) such that
\[
D(v_0,v_0,c)^{\,2m/c} \;=\; a_{P_0 w_S P_m}\,a_{P_m w_S P_0}.
\]
Hence
\[
\Spec\!\bigl((T_C)^{2m/c}\bigr) \;=\; \Spec\!\Bigl(\bigl(a_{P_0 w_S P_m}\,a_{P_m w_S P_0}\bigr)_r\Bigr).
\]

\item Applying Theorem~\ref{thm:eigenvalues} with \(I=S\setminus\{s_i,s_{i+1}\}\) (for any \((i,i+1)\in C\)) and \(J=I^{w_S}\), we obtain
\[
\Spec\!\Bigl(\bigl(a_{P_0 w_S P_m}\,a_{P_m w_S P_0}\bigr)_r\Bigr)
\;=\;
\left\{\,Q(w_I)^{-2}  \cdot q^{f_\chi}\ \text{with mult.\ } d_\chi n_\chi : \chi \in \Irr(G),\,n_\chi\neq 0\,\right\}.
\]
Thus each eigenvalue \(\lambda\) of \(T_C\) is a \((2m/c)\)-th root of a power of \(q\).

\item More precisely, recall that
\[
e_{P_I}\,\C G \;\simeq\; \bigoplus_\chi \bigl(e_{P_I} M_\chi\bigr)^{d_\chi}.
\]
The operator $T_C'$ acts on $e_{P_I}M_\chi$ with eigenvalues
\[
   \zeta \cdot \Bigl(Q(w_I)^{-2}\,q^{\,f_\chi}\Bigr)^{1/d},\qquad \zeta^d=1,
\]
where $d=2m/c$. We may therefore define
\[
   m_{C,\chi}(\zeta)
   \;=\;\text{the multiplicity of }
   \zeta \cdot \bigl(Q(w_I)^{-2}\,q^{\,f_\chi}\bigr)^{1/d}
   \text{ as an eigenvalue of $T_C$ on } e_{P_I}M_\chi.
\]
In particular,
\[
   \sum_{\zeta^d=1} m_{C,\chi}(\zeta) \;=\; n_\chi.
\]

\item The zeta factor attached to \(C\) is then
\[
Z_c\!\left(X_2(\mathcal B)|_C,u\right)\;=\;\prod_{\lambda\in \Spec(T_C)}
\frac{1}{1-\lambda\,u^{\,c}}=\;
\prod_{\chi:\,n_\chi\ne 0}\;\prod_{\zeta^d=1}
\frac{1}{\bigl(1-\zeta\,\left(Q(w_I)^{-2}\,q^{\,f_\chi}\right)^{1/d}\,u^{\,c}\bigr)^{\,m_{C,\chi}(\zeta)d_{\chi}}}.
\]
\end{enumerate}

\medskip
We have now proved the Main Theorem \ref{Mainthm}, in fact in the following refined form:

\begin{thm}[Main Theorem, refined form]\label{Mainthmv2}
Let \(G\) be a finite group of Lie type over \(\mathbb{F}_q\) with spherical building \(\mathcal{B}\). 
Let \(C\) be a type orbit of length \(c\) in \(X_2(\mathcal B)\), and let \(m\) be the integer associated with $C$ afforded by Luo’s decomposition (so \(c\mid 2m\)). Let $d=2m/c$ 
Fix any \((i,i+1)\in C\) and set \(I=S\setminus\{s_i,s_{i+1}\}\). Then:

\begin{enumerate}
\item For every eigenvalue \(\lambda\) of the $c$–step operator \(T_C\), the power \(\lambda^{2m/c}\) is an integer power of \(q\):
\[
\Spec\!\bigl((T_C)^{2m/c}\bigr) \;=\left\{\,Q(w_I)^{-2}  \cdot q^{f_\chi}\ \text{with mult.\ } d_\chi n_\chi : \chi \in \Irr(G),\,n_\chi\neq 0\,\right\}.
\]

\item Let $m_{C,\chi}(\zeta)$ be the multiplicity of the eigenvalue 
\(\zeta\cdot(Q(w_I)^{-2}q^{f_\chi})^{1/d}\) 
of $T_C$ acting on $e_{P_I}M_\chi$, where $\zeta^d=1$. Then

\[
   \sum_{\zeta^d=1} m_{C,\chi}(\zeta)=n_\chi
\]
and the zeta factor attached to \(C\) is
\[
Z_c\!\left(X_2(\mathcal B)|_C,u\right)=\;
\prod_{\chi:\,n_\chi\ne 0}\;\prod_{\zeta^d=1}
\frac{1}{\bigl(1-\zeta\,\left(Q(w_I)^{-2}\,q^{\,f_\chi}\right)^{1/d}\,u^{\,c}\bigr)^{\,m_{C,\chi}(\zeta)d_{\chi}}}.
\]
\end{enumerate}

Consequently, the global edge zeta function factors as
\[
Z\!\left(X_1(\mathcal B),u\right)=Z_c\!\left(X_2(\mathcal B),u\right)\;=\;\prod_{C} Z_c\!\left(X_2(\mathcal B)|_C,u\right),
\]
and all reciprocal poles of \(Z\) are roots of powers of \(q\).
\end{thm}

Once again, we record the notations:
\begin{itemize}
  \item $Q(w_I)$: $|Bw_IB/B|$.
  \item $\chi\in \Irr(G)$: an irreducible representation of $G$.
  \item $d_\chi=\dim(\chi)$: the degree of $\chi$.
  \item $n_\chi$: the multiplicity of $\chi$ in the permutation module $\C G e_{P_I}$.
  \item $f_\chi$: the reflection statistic
  \[
     f_\chi=\sum_{t\in \mathcal{R}}
     \Bigl(1+\frac{\chi_1(t)}{\chi_1(1)}\Bigr),
  \]
  with $\mathcal{R}$ the set of reflections in $W$ and $\chi_1$ the corresponding Weyl–group character.
\end{itemize}
\paragraph{Techniques to find $m_{C,\chi}$}  
The parameters $Q(w_I), \chi, d_\chi, n_\chi,$ and $f_\chi$ arise naturally from the representations of the Hecke algebra $\mathcal{H}=e_B\C G e_B$ and the group algebra $\C W$. To compute the multiplicities $m_{C,\chi}(\zeta)$, we deform $\mathcal{H}$ to $\C W$ by letting $q\to 1$.  

Recall that $P_0=P_I$ and $W_0=W_I$ are the parabolic subgroups of $G$ and $W$ corresponding to $I$, and that $m_{C,\chi}(\zeta)$ is the multiplicity of the eigenvalue  
\[
\zeta \cdot \bigl(Q(w_I)^{-2}\,q^{\,f_\chi}\bigr)^{1/d}
\]  
for the operator $T_C'$, where $T_C'=D(v_0,v_0,c)_\ell$ is left multiplication by $D(v_0,v_0,c)=a_{P_0uP_0}$ with $u=(w_0'w_1)\cdots(w_{c-1}'w_0)\in W$.  

By Tits’ Deformation Principle (\cite[\S8.1]{ShenAIM2024}, \cite[\S7.4,8.1]{geck_pfeiffer2000}), under $q\to 1$ we have $\mathcal{H}\to \C W$, $P_0 \to W_0$, and $D(v_0,v_0,c) \to a_{W_0uW_0} = \tfrac{1}{|W_0|}\sum_{x\in W_0uW_0}x$. The module $e_{P_I}M_\chi$ deforms to $e_{W_I}M_{1,\chi}$, where $M_{1,\chi}$ is the $\C W$-module corresponding to $\chi$. Thus $m_{C,\chi}(\zeta)$ equals the multiplicity of $\zeta$ as an eigenvalue of $a_{W_0uW_0}$ on $e_{W_0}M_{1,\chi}$.  

In particular,  
\[
\sum_\zeta m_{C,\chi}(\zeta)\,\zeta \;=\; \Tr\!\bigl((a_{W_0uW_0})_\ell,\,e_{W_0}M_{1,\chi}\bigr) \;=\; \Tr\!\bigl((a_{W_0uW_0})_\ell,M_{1,\chi}\bigr) \;=\; \chi_1(a_{W_0uW_0}),
\]  
where $\chi_1$ is the character of $M_{1,\chi}$. The second equality holds since the image of $(a_{W_0uW_0})_\ell$ on $M_{1,\chi}$ already lies in $e_{W_0}M_{1,\chi}$. More generally, for every integer $k$,  
\begin{equation}\label{eq:mult-identity}
\chi_1(a_{W_0uW_0}^k) \;=\; \sum_\zeta m_{C,\chi}(\zeta)\,\zeta^k,
\end{equation}
which provides a practical tool to find the multiplicities $m_{C,\chi}(\zeta)$.

\subsection{Luo’s decomposition theorem}\label{sec:Luo-new}
The following subsection concerns a recent result of Luo \cite[\S2]{luo2022} on
the decomposition of the longest element $w_S$ of a Weyl group.  
Luo originally formulated the theorem in the context of his work on the Casselman–Shahidi
conjecture, where it serves as a key combinatorial input for analyzing the singularity of
intertwining operators in exceptional types. In our setting the same theorem provides the
structural simplification needed to collapse long products of relative destination elements
to just a few double coset operators involving opposite parabolics.

Our presentation differs slightly from Luo’s.
Since Luo’s paper is recent and not yet published in journal form, and because we relied
heavily on this decomposition in our argument, we reproduce his proof in our notations. 

\medskip

Throughout, let $(W,S)$ be the Weyl group attached to a reduced root system $(\Phi,\Delta)$.  
For $w\in W$, set
\[
R(w):=w^{-1}(\Phi^-)\cap \Phi^+,
\]
the set of positive roots sent to negative by $w$. Then
\[\ell(w)=|R(w)|.\] 
For $\alpha\in\Phi$, write $\alpha>0$ if $\alpha\in\Phi^+$ and $\alpha<0$ if $\alpha\in\Phi^-$.  
For $\Delta'\subseteq \Delta$, denote by $(W_{\Delta'},S_{\Delta'})$ and $\Phi_{\Delta'}$ the corresponding parabolic subgroup and root subsystem.

\medskip

We begin with two standard lemmas on root sets and lengths.

\begin{lemma}\label{Lemma3.3.1}
Let $w_1,w_2\in W$.  
If $R(w_2)\subseteq R(w_1w_2)$, then $\ell(w_1w_2)=\ell(w_1)+\ell(w_2)$.  
If $R(w_2)\cap R(w_1w_2)=\varnothing$, then $\ell(w_1w_2)=\ell(w_1)-\ell(w_2)$.
\end{lemma}

\begin{proof}
If $R(w_2)\subseteq R(w_1w_2)$, any reduced expression of $w_2$ sits as a suffix of a reduced expression of $w_1w_2$, hence $\ell(w_1w_2)=\ell(w_1)+\ell(w_2)$.

If $R(w_2)\cap R(w_1w_2)=\varnothing$, then $R(w_2^{-1})\subseteq R(w_1)$: for $\alpha\in R(w_2^{-1})$ one has $w_2^{-1}\alpha<0<\alpha$, so $-w_2^{-1}\alpha\in R(w_2)$ and thus $-w_2^{-1}\alpha\notin R(w_1w_2)$; hence $-w_1\alpha>0$ and $\alpha\in R(w_1)$.  
Therefore $R(w_2)\subseteq R(w_1w_2w_2^{-1})$, and  $\ell(w_1)=\ell(w_1w_2)+\ell(w_2^{-1})=\ell(w_1w_2)+\ell(w_2^{-1})$.
\end{proof}

\begin{lemma}\label{Lemma3.3.2}
Let $\Delta'\subseteq \Delta$, let $\alpha\in\Delta'$ be simple, and let $w\in W$ with $w(\Delta'-\{\alpha\})\subseteq \Delta$.  
If $w\alpha<0$ (resp.\ $w\alpha>0$), then $w\beta<0$ (resp.\ $w\beta>0$) for every positive root $\beta\in \Phi_{\Delta'}-\Phi_{\Delta'-\{\alpha\}}$.
\end{lemma}

\begin{proof}
Put $\Delta'':=w(\Delta'-\{\alpha\})\subseteq \Phi^+$.  
If $w\alpha>0$, then $w\Delta'\subseteq \Phi^+$ and hence $w\beta>0$ for all positive $\beta\in \Phi_{\Delta'}$.

If $w\alpha<0$, write $w\alpha=\sum_{\gamma\in\Delta} c_\gamma\gamma$; since $w\alpha$ is linearly independent from $\Delta''$, there exists $\gamma_0\in \Delta\setminus\Delta''$ with $c_{\gamma_0}\neq 0$, and necessarily $c_{\gamma_0}<0$.  
For $\beta=\sum_{\delta\in\Delta'} c'_\delta \delta\in \Phi_{\Delta'}^+$ with $c'_\alpha>0$, the coefficient of $\gamma_0$ in $w\beta=\sum_{\delta\in\Delta'} c'_\delta\,w\delta$ equals $c_{\gamma_0}c'_\alpha<0$, whence $w\beta<0$.
\end{proof}

\medskip

We now state Luo's theorem in our notation.  
Given $s_0,s_1\in S$, define the bi-infinite sequence $(s_i)_{i\in\Bbb Z}$ by
\[
s_{n+1} \;=\; w_{S-\{s_n\}}\,s_{n-1}\,w_{S-\{s_n\}}
\]
For each $i\in\Bbb Z$ set
\[
w_i:=w_{S-\{s_i\}},\qquad
w'_i:=w_{S-\{s_i,s_{i+1}\}},
\]
the longest elements of the corresponding standard parabolic subgroups.

\begin{thm}[Luo]\label{LuoThm2}
There exists a unique $m\ge1$ such that
\[
w'_0\,w_S \;=\; (w'_0 w_1)(w'_1 w_2)\cdots (w'_{m-1} w_m),
\]
and the lengths add:
\[
\ell(w'_0 w_S) \;=\; \sum_{j=1}^{m}\, \ell(w'_{j-1} w_j).
\]
\end{thm}

\begin{proof}\textbf{Step 1: Set-up.}
Let $\alpha_i\in\Delta$ be the simple root corresponding to $s_i$, and set
\[
\Delta_i:=\Delta\setminus\{\alpha_i\},\qquad
\Delta_{i,i+1}:=\Delta\setminus\{\alpha_i,\alpha_{i+1}\}.
\]
Then $w'_i$ (resp.\ $w_{i+1}$) is the longest element of $W_{\Delta_{i,i+1}}$ (resp.\ $W_{\Delta_{i+1}}$), so
\begin{equation}\label{S0}
    w'_i(\Delta_{i,i+1})=-\Delta_{i,i+1},\qquad
w_{i+1}(\Delta_{i+1})=-\Delta_{i+1}, \qquad
(w'_i w_{i+1})\,\Delta_{i+1,i+2}=\Delta_{i,i+1}.
\end{equation}

Set $u_0:=w_S w'_0$ and inductively $u_{n+1}:=u_n\,(w'_n w_{n+1})$.
 Thus \[u_n=w_S\,w_1(w'_1 w_2)\cdots (w'_{n-1} w_n).\]  Our goal is to find some $m$ so that $u_m=e$.

\textbf{Step 2: Length formula.}
In this step, we deduce the length update
\begin{equation}\label{S3}
\ell(u_{n+1}) \;=\;
\begin{cases}
\ell(u_n)-\ell(w'_n w_{n+1}), & \text{if } u_n\alpha_n<0,\\[2pt]
\ell(u_n)+\ell(w'_n w_{n+1}), & \text{if } u_n\alpha_n>0.
\end{cases}
\end{equation}

Put
\(
v_n:=w_{n+1}w'_n=(w'_n w_{n+1})^{-1}.
\)
Then $u_{n+1}\,v_n \;=\; u_n$, and
\begin{equation}\label{S2}
R(v_n)=R(w_{n+1}w'_n)=\Phi_{\Delta_{n+1}}^+\setminus \Phi_{\Delta_{n,n+1}}^+.
\end{equation}

By (\ref{S0}), $u_{n+1}\Delta_{n+1,n+2}:=u_n\,(w'_n w_{n+1})\Delta_{n+1,n+2}=u_n\Delta_{n,n+1}$. Iterating gives
\begin{equation}\label{S1}
u_n\,\Delta_{n,n+1}
= u_0\,\Delta_{0,1}
= w_S\bigl(w'_0\Delta_{0,1}\bigr)
= w_S(-\Delta_{0,1})
\subset \Delta.
\end{equation}
Then, we may apply Lemma~\ref{Lemma3.3.2} with $w=u_n$, $\Delta'=\Delta_{n+1}$, $\alpha=\alpha_n$, and conclude that the sign of $u_n$ on the cone
$\Phi_{\Delta_{n+1}}^+\setminus \Phi_{\Delta_{n,n+1}}^+$ is controlled by the sign of $u_n\alpha_n$.

Combining with (\ref{S2}) gives the relation on reflected roots: \[
\begin{cases}
u_n\alpha_n>0 \ \Rightarrow\ R(v_n)\cap R(u_n)=\varnothing,\\
u_n\alpha_n<0 \ \Rightarrow\ R(v_n)\subseteq R(u_n).
\end{cases}\]
This, together with $u_n=u_{n+1}v_n$, implies (\ref{S3}) by Lemma~\ref{Lemma3.3.1}.

\textbf{Step 3: Existence of $m$ with $u_m=e$.}
We have $u_0\alpha_0=w_S w'_0\alpha_0<0$.  
As long as $u_n\alpha_n<0$, the length strictly decreases by (\ref{S3}).  
But $\ell(u_n)$ is bounded below. Hence there exists a minimal $m\ge1$ with $u_m\alpha_m>0$.

By minimality, $u_{m-1}\alpha_{m-1}<0$. Using $\alpha_{m+1}=-w_m\alpha_{m-1}$,
\[
u_m\alpha_{m+1}=u_{m-1}w'_{m-1}(-\alpha_{m-1}).
\]
Since $w'_{m-1}\alpha_{m-1}>0$ and $w'_{m-1}\alpha_{m-1}\in \Phi_{\Delta_m}^+\!\setminus\!\Phi_{\Delta_{m-1,m}}^+$, Lemma~\ref{Lemma3.3.2} (with $w=u_{m-1}$, $\Delta'=\Delta_m$, $\alpha=\alpha_{m-1}$) gives $u_{m-1}w'_{m-1}\alpha_{m-1}<0$, hence $u_m\alpha_{m+1}>0$.
Together with $u_m\alpha_m>0$ and (S1) (which implies $u_m\Delta_{m,m+1}\subset\Delta$), we conclude $u_m$ sends all simple roots to positive roots; therefore $u_m=e$.

Summing (\ref{S3}) for $n=0,\dots,m-1$ and using $u_m=e$ yields
\[
\ell(w'_0 w_S) \;=\; \sum_{j=1}^{m}\ell(w'_{j-1}w_j).
\]
Taking inverses of $u_m=e$ gives
\[
w'_0\,w_S \;=\; (w'_0 w_1)(w'_1 w_2)\cdots (w'_{m-1} w_m).
\]
Uniqueness of $m$ follows from the length additivity.
\end{proof}
\begin{rem}

In terms of our notations, Luo’s original statement (\cite{luo2022}) is
\[
   w_S w_1 \;=\; (w_m w'_{m-1})\cdots (w_2 w'_1),
\]
with length additivity. 
This is equivalent to
\(
w_1w_S  = (w'_1 w_2)\cdots(w'_{m-1}w_m)
\),
which is in turn equivalent (by multiplying $(w'_0 w_1)$ on the left) to the form
\[
   w'_0 w_S \;=\; (w'_0 w_1)(w'_1 w_2)\cdots (w'_{m-1} w_m).
\]
used above.  We adopt the latter form because it aligns more directly with the orbit structure of the geodesic edge graph and with our RDE computations, and because $2m$ is a period for the evolution of types.  
The values of $m$ for each irreducible type of $\Delta$ are listed in next section.
\end{rem}

%% file: sec4.tex
\section{Examples: Explicit Zeta Factors for Finite Buildings}\label{sec:4}
This section illustrates the general spectral and zeta–factor formulas developed in
Section~\ref{sec:3.2} by working out concrete cases for spherical buildings of
finite classical groups.  
We focus on the split families of type~$\mathbf{A}$ and type~$\mathbf{C}$, for which all Hecke
parameters equal~$q$ and the zeta factors admit closed expressions.
The type~$\mathbf{A}$ case recalls the results of~\cite{ShenAIM2024} in the present
framework, while the type~$\mathbf{C}$ case provides new explicit formulas for
$G=\Sp_{2n}(\F_q)$.  

Section~\ref{sec:4.1} reviews the structure of finite buildings and the
split versus non–split distinction.  
Sections~\ref{subsec:typeA}–\ref{subsec:typeC} apply the three–step
spectral method to types~$\mathbf{A}$ and~$\mathbf{C}$, and
Section~\ref{sec:4.4} records the type–orbit cycles and the half–cycle
integers~$m$ needed for all Coxeter types.
\subsection{Finite buildings: split and non–split cases}\label{sec:4.1}
By Tits’ classification of buildings, every finite spherical building of irreducible type and rank at least three arises from an algebraic group of Lie type over a finite field \cite[\S11.4]{Tits1974}.  
A complete list of such families is given in \cite[\S11.2]{Tits1974}.  
Our discussion applies uniformly to all of them.

Let $G$ be a connected reductive algebraic group defined over $\F_q$, and let $\B(G)$ be the spherical building associated with $G$, as in Section~\ref{sec:1}.  
In particular, groups such as $\GL_n(\F_q)$, though not absolutely simple, yield the same building as $\PGL_n(\F_q)$ and are therefore included in our framework.  

The Coxeter system $(W,S)$ underlying $G$ determines the type of apartments of $\B(G)$.  
Recall from §\ref{sec:computing-RDE} that for each $s\in S$, the \emph{thickness parameter}
\[
  Q(s)\;=\;\#\{\text{chambers adjacent to a given chamber via a panel of type $s$}\}
  \;=\; |BsB/B|
\]
records the local branching number at panels of type $s$.  
For $w\in W$, we set $Q(w)=\prod_{s\in\mathrm{red}(w)}Q(s)$, independent of the choice of reduced expression.  

\paragraph{Split versus non–split}
The collection $\{Q(s)\}_{s\in S}$ detects whether $G$ is split over $\F_q$.  
If $G$ is split (e.g.\ $G=\GL_n,\Sp_{2n},\SO_{2n}$), then all $Q(s)$ equal the field size $q$, so the Iwahori–Hecke algebra has equal parameters.  
In this case,
\[
Q(w)=q^{\ell(w)}, \qquad Q(w_I)=q^{\ell(w_I)}.
\]

For twisted or non–split groups, the parameters $Q(s)$ may vary across different conjugacy classes of simple reflections.  
An important infinite family is provided by the unitary groups ${}^2A_n(q)=\mathrm{PGU}_{n+1}(q)$ ($n\ge 2$), defined by a non–degenerate hermitian form.  
Their building is the flag complex of the associated polar space, of type $\mathbf{C}_{\lfloor (n+1)/2 \rfloor}$ \cite[\S11.2.2]{Tits1974}.  

Our algorithm applies uniformly in both split and non–split settings.  

\paragraph{Three–step summary}
We now summarize the procedure for computing the edge eigenvalues and edge zeta functions of $\B(G)$.  

\emph{Step 1. Orbit decomposition.}  
Decompose the set $S\times S-\Delta(S)$ of directed edge types into orbits $C$, and write $c=|C|$.  
This gives a decomposition of $X_2(B)$ into type–orbit components $X_2(B)\mid_C$.  
For each orbit, choose a starting edge $v_0$ with stabilizer $P_0=\Stab(v_0)$.  
Form the $c$–step operator $T_C$ on $\C[G]e_{P_0}$, given by right multiplication by $D(v_0,v_0,c)=a_{P_0uP_0}$, with $u=(w_0'w_1)\cdots(w_{c-1}'w_0)\in W$ (see the end of §\ref{sec:transitive-RDE}).

Then
\[
Z_c\!\bigl(X_2(B),u\bigr)=\prod_{C} Z_C(u),
\qquad
Z_C(u)=\prod_{\lambda\in\Spec(T_C)} \frac{1}{1-\lambda\,u^{\,c}}.
\]

\emph{Step 2. Collapse to $D(v_0,v_0,2m)$.}  
Luo’s decomposition produces an integer $m$ (explicitly determined in §\ref{sec:4.4}) with $c\mid 2m$, such that
\[
(T_C)^{2m/c}=D(v_0,v_0,2m)_r.
\]
Its spectrum is
\[
\Spec\!\bigl(D(v_0,v_0,2m)\bigr)
   = \biguplus_{\chi\in\Irr(G)}
      \bigl\{\,q^{\,f_\chi}\,Q(w_I)^{-2}\bigr\}^{\times d_\chi n_\chi},
\]
with $f_\chi,d_\chi,n_\chi$ as in §\ref{subsec:main-conclusion}.  

\emph{Step 3. Root selection.}  
The eigenvalues of $T_C$ are the $(2m/c)$–th roots of these collapsed values.  
Their multiplicities are determined by equation~\eqref{eq:mult-identity} (see §\ref{subsec:main-conclusion}).  
Consequently, by Main Theorem~\ref{Mainthmv2},
\[
Z_c\!\left(X_2(\mathcal B)\!\mid_C,u\right)
   = \prod_{\chi:\,n_\chi\ne 0}\;\prod_{\zeta^d=1}
     \frac{1}{\bigl(1-\zeta\,(Q(w_I)^{-2}q^{\,f_\chi})^{1/d}\,u^c\bigr)^{m_{C,\chi}(\zeta)\,d_\chi}}.
\]

In the next subsections we first recall the type~$\mathbf{A}$ computation, and then work out the new formulas for $G=\Sp_{2n}(\F_q)$.

\subsection{Type $\mathbf{A}$: recalling the formulas}\label{subsec:typeA}
We summarize the zeta functions for buildings of type $\mathbf{A}_n$ from \cite{ShenAIM2024} using the present three–step method; we refer to that paper for the details. 
By Tits' classification \cite[\S11.2]{Tits1974}, all buildings of type $\mathbf{A}_n$ with $n\ge 3$ are associated with \(G=\GL_{n+1}(\F_q)\).

Let \(G=\GL_n(\F_q)\), whose spherical building is of type \(A_{n-1}\). It admits the concrete description:
\begin{itemize}
  \item Vertices of type \(i\) correspond to \(i\)-dimensional subspaces of \(\F_q^n\),
        equivalently to cosets of maximal parabolics \(P_{S-\{s_i\}}\).
  \item Edges correspond to nested subspaces \(U\subset V\), equivalently to cosets \(P_{S-\{s_i,s_j\}}\).
  \item Chambers are maximal flags \(0\subset V_1\subset \cdots \subset V_{n-1}\subset \F_q^n\).
\end{itemize}
The Weyl group is \(W=S_n\), generated by the adjacent transpositions \(s_i=(i\ i{+}1)\) for \(1\le i\le n-1\).
Geometrically, the building is the flag complex of \(\F_q^n\), and every thickness parameter equals \(q\).

\medskip\noindent
We now recover the closed formulas for zeta functions in type \(A\) via the three–step scheme.

\medskip\noindent
\textbf{Step 1. Cycles and the parameter \(2m\).}
Write \(n=i+j+k\) with \(i\ge j\) and \(i\ge k\).
A type orbit \(C\) is determined by the first two dimensions \((i,i{+}j)\).
Starting from the edge \(i\to i{+}j\) and iterating the next–type map yields the orbit
\[
C=\{(i,i{+}j),\,(i{+}j,j),\,(j,j{+}k),\,(j{+}k,k),\,(k,k{+}i),\,(k{+}i,i)\}.
\]
Generically \(c=|C|=6\); in the balanced case \(i=j=k\) one has \(c=2\).
In type \(A\), Luo’s decomposition always gives \(2m=6\).
Fix \(v_0=\langle e_1,\dots,e_i\rangle \subset \langle e_1,\dots,e_{i+j}\rangle\) and let \(P_0=\Stab(v_0)\).
With \(T_C=D(v_0,v_0,c)_r\), 
\[
Z_C(u)\;=\;\prod_{\lambda\in\Spec(T_C)} \frac{1}{1-\lambda\,u^{c}}.
\]

\medskip\noindent
\textbf{Step 2. Collapsed spectrum.}
Let \(P_0=P_\mu\) with row–type \(\mu=(\mu_1,\mu_2,\mu_3)=(i,j,k)\) in descending order.
Then
\[
(T_C)^{2m/c}\;\sim\;D(v_0,v_0,2m),\qquad
\Spec\!\bigl(D(v_0,v_0,2m)\bigr)
=\biguplus_{\chi\in\Irr(G)}
\bigl\{\,q^{\,f_\chi}\,Q(w_I)^{-2}\bigr\}^{\times d_\chi n_\chi}.
\]

The irreducible constituents of $\C G e_{P_\mu}$ are denoted $\chi_\lambda$, 
parametrized by partitions $\lambda\vdash n$ that dominate $\mu$ ($\lambda\succeq\mu$).  

Furthermore:

\begin{itemize}
  \item \(n_\lambda=\) multiplicity of \(\chi_\lambda\) in \(\C Ge_{P_\mu}\), and equals the Kostka number \(K_{\lambda,\mu}\).
  \item 
  \(Q(w_I)=q^{\ell(w_I)}\), with \(\ell(w_I)=\mathrm{wt}_r(\mu):=\sum_i \binom{\mu_i}{2}\) the row weight of $\mu$.
  \item \(f_\lambda=\displaystyle \frac{n(n-1)}{2}+\kappa_{\lambda}\), where \(\kappa_{\lambda}=\sum_i \binom{\lambda_i}{2}-\sum_j \binom{\lambda'_j}{2}\), with \(\lambda'\) the conjugate partition of $\lambda$.
  \item \(d_\lambda=\displaystyle q^{\,n(\lambda)}\frac{\prod_{i=1}^n (q^i-1)}{\prod_{b\in\lambda}(q^{h(b)}-1)}\) (the \(q\)-hook formula), and $n(\lambda)=\sum_{i\geq 1} (i-1)\lambda_i$.
\end{itemize}
Hence
\[
\Spec\!\bigl(D(v_0,v_0,2m)\bigr)
=\biguplus_{\lambda\vdash n}\Bigl\{\,q^{\,f_\lambda-2\,\mathrm{wt}_r(\mu)}\Bigr\}^{\times d_\lambda K_{\lambda,\mu}}.
\]

\medskip\noindent
\textbf{Step 3. Root selection.}
The eigenvalues of \(T_C\) are the \((2m/c)\)-th roots of the collapsed values above.
In type \(A\), \(2m=6\) (See for example \S\ref{sec:4.4} below).
\begin{itemize}
  \item If \(c=6\) (generic), then \(2m/c=1\), so no further root extraction is needed.
  \item If \(c=2\) (balanced), then \(2m/c=3\): we need to take cubic roots of the eigenvalues $q^{\,f_\lambda-2\,\mathrm{wt}_r(\mu)}$.
        Multiplicities are as in Table~8 of \cite[\S8.3]{ShenAIM2024}:
        the \(K_{\lambda,\mu}\) eigenvalues split as evenly as possible among \(1,\omega,\omega^2\) (\(\omega^3=1\)).
        Thus
        \[m(1)+m(\omega)+m(\omega^2)=K_{\lambda,\mu},\qquad m(\omega)=m(\omega^2),
        \]
        and \(m(1)\) differs from \(m(\omega)\) by at most \(1\), with  \(m(1)-m(\omega)\equiv K_{\lambda,\mu}\bmod 3\).
\end{itemize}

Therefore we recover the following closed formulas in type \(A\) \cite[\S7.2, \S8.4]{ShenAIM2024}.

\begin{thm}[Type $\mathbf{A}$, generic 6–partite components]\label{thm:A-generic}
Let \(C\) be a generic type–orbit (so \(c=6\)).
Then
\[
  Z_c\bigl(X_2(B)|_C,u\bigr)
  \;=\;
  \prod_{\lambda\vdash n}
  \frac{1}{\bigl(1 - q^{\,f_\lambda - 2\,\mathrm{wt}_r(\mu)}\,u^{6}\bigr)^{d_\lambda K_{\lambda,\mu}}}.
\]
Here \(d_\lambda\) is given by the \(q\)–hook length formula, \(K_{\lambda,\mu}\) is the Kostka number, and \(f_\lambda\) is as above.
\end{thm}

\begin{thm}[Type $\mathbf{A}$, balanced 2–partite components]\label{thm:A-balanced}
Let \(C\) be a balanced type–orbit (so \(c=2\)).
Then
\[
  \frac{1}{Z_c\!\bigl(X_2(B)|_C,u\bigr)}
  \;=\;
  \prod_{\lambda \vdash n}\;
  \prod_{j=0}^2
  \Bigl(1 - \omega^j\,q^{(f_\lambda-2\,\mathrm{wt}_r(\mu))/3}\,u^{2}\Bigr)^{d_\lambda\,(K_{\lambda,\mu}+\epsilon_{\lambda,j})/3},
\]
where \(\omega=e^{2\pi i/3}\), and the correcting integers satisfy
\(\epsilon_{\lambda,0}\in\{0,\pm2\}\), \(\epsilon_{\lambda,1}=\epsilon_{\lambda,2}\in\{0,\pm1\}\),
chosen so that the exponents sum to \(d_\lambda K_{\lambda,\mu}\) and match the balanced split described above.
\end{thm}

\subsection{Type $\mathbf{C}$: explicit formulas}\label{subsec:typeC}

In this subsection we consider the spherical building attached to the symplectic group 
$G=\Sp_{2n}(\F_q)$, which is split of type $\mathbf{C}_n$. 
We recall the geometric description of its building, and present the explicit formulas for the edge zeta functions by computing explicitly its invariants.

\subsubsection{Geometry and zeta factors set-up}

\paragraph{Group and BN--pair}
Let $V=\F_q^{2n}$ with a nondegenerate alternating form $\langle\ ,\ \rangle$ and
\[
  G=\Sp_{2n}(\F_q)=\{g\in\GL(V)\mid \langle gv,gw\rangle=\langle v,w\rangle\}.
\]
Fix a symplectic basis $\mathcal{E}=\{e_1,\dots,e_n,f_n,\dots,f_1\}$ with $\langle e_i,f_j\rangle=\delta_{ij}$, and set
$V_i=\langle e_1,\dots,e_i\rangle$.
The standard maximal torus and Borel are
\[
  T=\{\diag(t_1,\dots,t_n,t_n^{-1},\dots,t_1^{-1})\},\qquad
  B=\Stab_G(0\subset V_1\subset\cdots\subset V_n),
\]
and $N=\Stab_G(T)$ is the group of monomial matrices in $G$ with respect to \ $\mathcal{E}$.
Then $(B,N)$ is a BN--pair for $G$, with $|B|=q^{n^2}(q-1)^n$ and
$|G|=q^{n^2}\prod_{i=1}^n(q^{2i}-1)$.

\paragraph{Weyl group}
The Weyl group is the hyperoctahedral group
\[
  W=W(C_n)=N/T=\{\,\sigma:\{\pm1,\dots,\pm n\}\to\{\pm1,\dots,\pm n\}\mid \sigma(-i)=-\sigma(i)\,\},
\]
acting by signed permutations on the pairs $\{\langle e_i\rangle,\langle f_i\rangle\}$, where we identify
$\langle e_i\rangle$ with $+i$ and $\langle f_i\rangle$ with $-i$.
Choose simple reflections
\[
  s_i=(i\,\,i{+}1)(-i\,\, -(i{+}1))\ (1\le i<n),\qquad s_n=(n\,\, -n).
\]
For each $s\in S$ one has $|BsB|=q\,|B|$ since in the split case $Q(s)=q$.

\paragraph{Building and oppositeness}
The spherical building $\B(G)$ has
\begin{itemize}
  \item vertices of type $i$ parametrizing $i$--dimensional totally isotropic subspaces (equivalently $G/P_i$ with $P_i=\Stab_G(V_i)$);
  \item edges given by inclusions $U\subset U'$ of isotropic subspaces with $\dim U=i<j=\dim U'$;
  \item chambers are the maximal isotropic flags $0\subset V'_1\subset\cdots\subset V'_n$.
  \item Oppositeness for vertices: This is type-preserving: $U,U'$ of dimension $i$ are opposite iff the form induces a nondegenerate pairing
$U\times U'\to\F_q$, equivalently $g^{-1}h\in P_{S-\{i\}}\,w_0\,P_{S-\{i\}}$ for $U=gV_i$, $U'=hV_i$.
\end{itemize}

\emph{Remark.} Let $V_{n+1}=\langle e_1,...,e_n,f_n\rangle$.
Replacing $V_i$ inside $V_{i-1}\subset(\cdot)\subset V_{i+1}$ shows each $i$--panel meets exactly $q{+}1$ chambers.  This shows also that $|Bs_iB|=q|B|$.

\paragraph{Geodesic paths}
Recall from \S\ref{sec:1} that three vertices
\((aP_{S-\{r\}},\,bP_{S-\{s\}},\,cP_{S-\{t\}})\) form a geodesic path if the first and third are opposite in 
the link of the middle one. 

In type $\mathbf{C}$, this condition translates into an alternating pattern of inclusions and complements of isotropic subspaces.  
Such paths can be visualized as
\[
V_1 \subset V_2 \supset V_3 \subset V_4 \supset \cdots,
\]
with the condition that for each integer $k$, $V_{2k-1}\oplus V_{2k+1} = V_{2k}$ and that the symplectic form 
$\langle \ , \ \rangle$ induces a nondegenerate pairing on the  quotients 
$V_{2k}/V_{2k+1}$ and $V_{2k+2}/V_{2k+1}$. 

For example, the following is a geodesic cycle of length 8.
\begin{align*}
   &\langle e_1,\dots,e_i\rangle 
      \;\subseteq\; \langle e_1,\dots,e_j\rangle
      \;\supseteq\; \langle e_{i+1},\dots,e_j\rangle
      \;\subseteq\; \langle f_1,\dots,f_i,e_{i+1},\dots,e_j\rangle \\
   &\supseteq\; \langle f_1,\dots,f_i\rangle 
      \;\subseteq\; \langle f_1,\dots,f_j\rangle
      \;\supseteq\; \langle f_{i+1},\dots,f_j\rangle
      \;\subseteq\; \langle e_1,\dots,e_i,f_{i+1},\dots,f_j\rangle\\ &
      \;\supseteq\; \langle e_1,\dots,e_i\rangle \subseteq .....
\end{align*}

\paragraph{Type orbits}

Label the $\mathbf{C}_n$ Dynkin diagram as $1-2-\cdots-(n-1)=n$.  
Fix integers $i,j\ge 1$ with $i+j\le n$.  
A geodesic cycle  beginning with isotropic subspaces of dimensions $(i,\,i+j)$ has type orbit 
\[
C=\{(i,i{+}j),\,(i{+}j,j),\,(j,i{+}j),\,(i{+}j,i)\},
\]
so $C$ has length $c=4$ when $j\neq i$, and length $c=2$ when $j=i$.

By Luo’s decomposition (see also \S\ref{sec:4.4}),  $2m=8$ in type $\mathbf{C}$, hence $d=2m/c\in\{2,4\}$.

\paragraph{Zeta functions and spectra}
Fix an orbit $C$ starting at $(i,i+j)$ and let $I=S\setminus\{s_i,s_{i+j}\}$, $P_0=P_I$.
By the Three–Step Summary of \S\ref{sec:4.1}:

Let $T_C$ be the right multiplication by $D(v_0,v_0,c)=a_{P_0uP_0}$ on $\C[G]e_{P_0}$, with $u=(w_0'w_1)\cdots(w_{c-1}'w_0)\in W$. Then
\[
(T_C)^{\,2m/c}=D(v_0,v_0,2m)_r,\qquad
\Spec D(v_0,v_0,2m)=\biguplus_{\chi\in\Irr(G)}
\bigl\{\,q^{\,f_\chi}\,Q(w_I)^{-2}\bigr\}^{\times d_\chi n_\chi}.
\]

Therefore the eigenvalues of $T_C$ are the $d$th roots of these collapsed values, distributed according to the root–selection multiplicities $m_{C,\chi}(\zeta)$ of \S\ref{sec:4.1}:
\[
Z_c\!\bigl(X_2(\B)\!\mid_C,u\bigr)
=\prod_{\chi:\,n_\chi\ne0}\ \prod_{\zeta^{\,d}=1}
\frac{1}{\Bigl(1-\zeta\,(Q(w_I)^{-2}q^{\,f_\chi})^{1/d}\,u^{\,c}\Bigr)^{m_{C,\chi}(\zeta)\,d_\chi}}.
\]
Here $c\in\{2,4\}$, $d=2m/c\in\{4,2\}$ with $2m=8$ in type $\mathbf{C}$; the invariants $f_\chi,d_\chi,n_\chi,m_{C,\chi}(\zeta)$ are as in \S\ref{subsec:main-conclusion}, and their explicit values will be obtained from the representation theory of  $W_n$ and  $G$, which we introduce subsequently.

\begin{rem}[Other buildings of type $\mathbf{C}$]
Beyond the split symplectic case $G=\Sp_{2n}(\F_q)$, the same spherical buildings of type $\mathbf{C}$ arise from, among others, two other classical families \cite[\S11.2]{Tits1974}.

(i) For odd orthogonal groups $G=\mathrm{O}_{2n+1}(\F_q)$ (type $\mathbf{B}_n$), the flag complex of the polar space of a nondegenerate quadratic form is a building of type $\mathbf{C}_n$; in particular, the Weyl group is again the hyperoctahedral group $W(C_n)$.

(ii) For unitary groups $G=\mathrm{PGU}_{n+1}(\F_{q^2})$ (twisted type ${}^{2}\!A_n$), the flag complex of the hermitian polar space is a building of type $\mathbf{C}_{\lfloor (n+1)/2\rfloor}$.

Along a fixed type–orbit $C$ in the geodesic edge graph, the invariants
\(\ell(w_I),\,f_\chi,\,n_\chi,\) and \(m_{C,\chi}(\zeta)\) depend only on the Weyl group \(W(C_n)\) and on the chosen orbit \(C\); only the degree \(d_\chi\) depends on the ambient group representation.
For the natural correspondence between labels \((\lambda,\mu)\vdash n\) on \(W(C_n)\) and the representations of $G=\Sp_{2n}(\F_q)$ or \(\mathrm{O}_{2n+1}(\F_q)\), the associated representations for \(\Sp_{2n}(\F_q)\) and \(\mathrm{O}_{2n+1}(\F_q)\) have the same dimensions \(d_\chi\) \cite[\S13.8]{carter1985}.
Consequently, the edge–geodesic zeta factors for \(\B(\Sp_{2n}(\F_q))\) and \(\B(\mathrm{O}_{2n+1}(\F_q))\) coincide, although the two buildings are not isomorphic when \(n\ge3\).
\end{rem}

\subsubsection{Spectral inputs for type $\mathbf{C}$}\label{subsec:typeC-spectral}

\paragraph{(1) Representations of $W_n$}
Let $W_n=W(C_n)=C_2^{\,n}\rtimes S_n$ with $C_2^{\,n}=\langle\tau_1,\dots,\tau_n\rangle$ acting by sign changes and $S_n$ permuting coordinates.
Its irreducible representations are indexed by \emph{bipartitions} $(\lambda,\mu)\vdash n$, i.e. pairs of partitions with $|\lambda|+|\mu|=n$. Write $n_+=|\lambda|$, $n_-=|\mu|$, and define the linear character $\rho_{n_+,n_-}$ of $C_2^{\,n}$ by

\[
\rho_{n_+,n_-}(\tau_i)=
\begin{cases}
+1,&1\le i\le n_+,\\
-1,&n_+< i\le n.
\end{cases}
\]
Its stabilizer in $S_n$ is $S_{n_+}\times S_{n_-}$, so it extends to $C_2^{\,n}\rtimes(S_{n_+}\times S_{n_-})$ by acting trivially on $S_{n_+}\times S_{n_-}$.

Let $S^\lambda, S^\mu$ denote the Specht modules of the symmetric group $S_{n_+}, S_{n_-}$, so their characters are $\chi_{\lambda,1}$ and $ \chi_{\mu,1}$ in \S\ref{subsec:typeA}. 
Inflate $S^\lambda\boxtimes S^\mu$ to $C_2^{\,n}\rtimes(S_{n_+}\times S_{n_-})$ (trivial on $C_2^{\,n}$) and set
\[
\widetilde S^{\lambda,\mu}=\rho_{n_+,n_-}\otimes (S^\lambda\boxtimes S^\mu),\qquad
S^{\lambda,\mu}=\Ind_{\,C_2^{\,n}\rtimes(S_{n_+}\times S_{n_-})}^{W_n}\widetilde S^{\lambda,\mu}.
\]
By Clifford theory, $S^{\lambda,\mu}$ is irreducible and every irreducible arises this way. Let $\chi^{\lambda,\mu}=\chi^{\lambda,\mu}_1$ be its character.

\paragraph{(2) Deformation from $\C W_n$ to $\mathcal H_q$ and $G$}
Let $e_B=|B|^{-1}\sum_{b\in B}b$ and $\mathcal H_q=e_B\,\C G\,e_B$. Then $\mathcal H_q$ is a Hecke algebra with basis $\{T_w : w\in W\}$ given by
\[
T_w=a_{BwB} \;:=\; \tfrac{1}{|B|}\sum_{x\in BsB} x,
\]
It is generated by $\{T_s : s\in S\}$ subject to the braid and quadratic relations $T_s^2=(q-1)T_s+q$. We compare the representations:

\smallskip
\begin{enumerate}[label=(\alph*)]
\item \emph{$\Irr(\mathcal H_q)$ vs.\ $\Irr(W_n)$}
The family $\{\mathcal H_q\}_{q\neq 0}$ is a flat deformation of $\C W_n$: for $q=1$ one has $\mathcal H_1= \C W_n$. 
By Tits’ deformation theorem \cite[§8.1.7]{geck_pfeiffer2000}, there is a correspondence between irreducible representations
\[
   \Irr(\mathcal H_q)\;\longleftrightarrow\;\Irr(W_n),
\]
preserving dimensions and inner products. 

\emph{Remark.} There exists a noncanonical algebra isomorphism $\mathcal H_q\simeq \C W_n$, realizing this deformation explicitly.

\smallskip
\noindent
\item \emph{$\Irr(G)$ vs.\ $\Irr(\mathcal H_q)$}\;
Note that $\mathcal H_q=e_B\C Ge_B$ is a Hecke algebra for $\C G$. The irreducible representations of $G$ contained in $\C Ge_B$ corresponds to  $\Irr(\mathcal H_q)$. More precisely:

If $M$ is a simple $\C G$–module contained in $\C G e_B$, then $e_BM$ is a simple $\mathcal H_q$–module, and every simple $\mathcal H_q$–module arises this way \cite[II, §11D]{curtis1981}.
On modules and characters, the correspondence is
\[
   \Irr(\C G)_{\subseteq \C Ge_B}\;\longleftrightarrow\;\Irr(\mathcal H_q): \qquad M\longleftrightarrow e_BM, \quad \chi \;\longleftrightarrow\; \chi|_{\mathcal H_q}.
\]
\end{enumerate}

\noindent Therefore,
\(\Irr(\C G)_{\subseteq \C Ge_B}\;\longleftrightarrow\; \Irr(\mathcal H_q)\;\longleftrightarrow\;\Irr(W_n)\).
We denote by $\chi^{\lambda,\mu}_q \in\Irr(\C G)$ the character corresponding to $\chi^{\lambda,\mu}_1\in \Irr(W_n)$.

This identification extends $\Z$--linearly and preserves parabolic induction \cite[\S9.1.9]{geck_pfeiffer2000}:  $1_{P_I}^G=\C Ge_{P_I} \leftrightarrow 1_{W_I}^{W_n}$. 

\paragraph{(3) Induction and permutation characters}

\smallskip
\begin{enumerate}[label=(\alph*)]
\item \emph{Induction and Kostka products.}  
For $S_n$, induction satisfies
\[
\Ind_{S_{n_1}\times S_{n_2}}^{S_{n_1+n_2}}
\bigl(\chi^{\lambda}\boxtimes \chi^{\mu}\bigr)
=\sum_{\nu} c^{\nu}_{\lambda,\mu}\,\chi^{\nu}, 
\quad 
\Ind_{S_{\alpha_1}\times\cdots\times S_{\alpha_r}}^{S_n}\mathbf{1}
=\sum_{\lambda\vdash n} K_{\lambda,(\alpha_1,\dots,\alpha_r)}\,\chi^\lambda,
\]
with Littlewood–Richardson coefficients $c^\nu_{\lambda,\mu}$ and Kostka numbers $K_{\lambda,\mu}$.  
Since $\mathbf{1}_{S_i}=\chi^{(i)}$, one has 
$K_{\lambda,(m_1,\dots,m_r)}=c^\lambda_{(m_1),\dots,(m_r)}$ for one–row parts.  

For $W_n$, Lemma~6.1.3 of \cite{geck_pfeiffer2000} gives, for $(\lambda_1,\mu_1)\vdash n_1$, $(\lambda_2,\mu_2)\vdash n_2$,
\[
\Ind_{W_{n_1}\times W_{n_2}}^{W_{n_1+n_2}}
\bigl(\chi^{(\lambda_1,\mu_1)}\boxtimes \chi^{(\lambda_2,\mu_2)}\bigr)
=\sum_{(\nu_1,\nu_2)} 
c^{\nu_1}_{\lambda_1,\lambda_2}\,
c^{\nu_2}_{\mu_1,\mu_2}\,
\chi^{(\nu_1,\nu_2)},
\]
so the $+$ and $-$ parts carry LR coefficients independently. Iterating and using that
$K_{\lambda,(m_1,\dots,m_r)}=c^\lambda_{(m_1),\dots,(m_r)}$ yields
\[
\Ind_{\prod_{t=1}^r W_{i_t+j_t}}^{W_{\sum_t(i_t+j_t)}}
\ \bigboxtimes_{t=1}^r \chi^{(i_t),(j_t)}
=\sum_{\lambda,\mu}
K_{\lambda,(i_1,\dots,i_r)}\;
K_{\mu,(j_1,\dots,j_r)}\;\chi^{(\lambda,\mu)}.
\]
A further useful identity (Lemma~6.1.4 of \cite{geck_pfeiffer2000}) is
\(
1_{S_n}^{W_n}=\sum_{i=0}^n \chi^{((i),(n-i))}.
\)

\item \emph{Decomposition of parabolic modules.}  
For $W_I=S_i\times S_j\times W_c$ with $i+j+c=n$,
\begin{align*}
1_{W_I}^{W_n}
&= \Ind_{S_a\times S_b\times W_c}^{W_n}
   \bigl(1_{S_a}\boxtimes 1_{S_b}\boxtimes 1_{W_c}\bigr) = \Ind_{W_a\times W_b\times W_c}^{W_n}
   \Bigl(\Ind_{S_a}^{W_a}\!1\ \boxtimes\ \Ind_{S_b}^{W_b}\!1\ \boxtimes\ 1_{W_c}\Bigr) \\
&= \Ind_{W_i\times W_j\times W_c}^{W_n}
   \Bigl(\;\bigoplus_{a=0}^{i}\chi^{((a),(i-a))}
          \ \boxtimes\ \bigoplus_{b=0}^{j}\chi^{((b),(j-b))}
          \ \boxtimes\ \chi^{((c),\varnothing)}\Bigr) \\
&\cong \sum_{\lambda,\mu}
     K_{\lambda,(a,b,c)}\,
     K_{\mu,(i-a,j-b)}\,
     \chi^{(\lambda,\mu)}.
\end{align*}
Thus the multiplicity of $\chi^{(\lambda,\mu)}$ in $1_{W_I}^{W_n}$ is
\[
n_{(\lambda,\mu)}=\sum_{a=0}^{i}\sum_{b=0}^{j}
K_{\lambda,(a,b,c)}\,K_{\mu,(i-a,j-b)},
\]
where $|\lambda|=a+b+c$, $|\mu|=i+j-a-b$. By the correspondence of (2), the same multiplicity holds for $\chi^{\lambda,\mu}_q$ in $\C Ge_{P_I}$.
\end{enumerate}

\paragraph{(4) Special values: degrees and special elements}
For an irreducible representation $\chi^{\lambda,\mu}$ of $W_n$ (or its deformation
$\chi^{\lambda,\mu}_q$ of $G$), we require several explicit evaluations. 
Let 
\[
   \widetilde S^{\lambda,\mu}=\rho_{n_+,n_-}\otimes(S^\lambda\boxtimes S^\mu)
\]
be the $H:=C_2^{\,n}\rtimes(S_{n_+}\times S_{n_-})$–module from part (a), with representation 
$\rho_{\widetilde S}$. The irreducible $W_n$–module $S^{\lambda,\mu}$ is obtained as 
\(\Ind_{H}^{W_n}\widetilde S^{\lambda,\mu}\), with representation $\rho_{S}$.

Now let \(x=\sum_{g\in W_n} c_g g \in Z(\C W_n)\), and denote by 
\(x'=\sum_{g\in H} c_g g \in \C H\) its restriction (projection) to $\C H$. Then $x'\in Z(\C H)$, and
by Schur’s lemma, $x$ acts on $S^{\lambda,\mu}$ by some scalar $c$, while $x'$ acts on 
$\widetilde S^{\lambda,\mu}$ by a scalar $c'$. Since 
\(S^{\lambda,\mu} = \C W_n \otimes_{\C H} \widetilde S^{\lambda,\mu}\), for any 
$m \in \widetilde S^{\lambda,\mu}$ we have
\[
c\,(1\otimes m)
= x(1\otimes m) 
= \sum_{g\in W_n} c_g\,(g\otimes m) 
= 1\otimes (x'm)\;+\;\text{(terms in other copies $g\otimes \widetilde S^{\lambda,\mu}$)}.
\]
But \(1\otimes (x'm)=c'\,(1\otimes m)\), hence $c=c'$. 
Therefore $x$ acts on $S^{\lambda,\mu}$ by the same scalar as $x'$ acts on 
$\widetilde S^{\lambda,\mu}$. We will use this scalar equality in (b) and (c) below.

\begin{enumerate}[label=(\alph*)]
\item \emph{Degrees.} 
The Weyl group representation degree is
\(
   \chi^{\lambda,\mu}_1(1)=\binom{n}{|\mu|}\,d_{\lambda,1}\,d_{\mu,1},
\)
where $d_{\nu,1}=\dim S^\nu$ is the dimension of the symmetric group module $S^\nu$.  
The corresponding $G$–degree is the \emph{generic degree polynomial}
\(
   d_{\lambda,\mu}(q)=\chi^{\lambda,\mu}_q(1).
\)

For bipartitions $(\lambda,\mu)\vdash n$ these polynomials are given by the \emph{hook–cohook formula} \cite[§13.8]{carter1985}.  
To describe it, one forms the \emph{defect–1 symbol} $S=(X,Y)$ of $(\lambda,\mu)$ (\cite[§11.4.2]{carter1985}):
\[
X=\{\lambda_i-i+k : 1\leq i\leq k\},\qquad
Y=\{\mu_j-j+(k-1):1\leq j\leq k-1\},
\]
for any integer $k\geq \max{(\ell(\lambda),\ell(\mu)+1)}$. 
A \emph{hook} is a pair $(b,c)$ with $0\leq b<c$ such that $c\in X$ and $b\notin X$, or $c\in Y$ and $b\notin Y$; a \emph{cohook} is a pair $(b,c)$ with $0\leq b<c$ such that $c\in X$ and $b\notin Y$, or $c\in Y$ and $b\notin X$.  
Let
\[
a(S)=\sum_{\{b,c\}\subset X\sqcup Y}\min(b,c)
      - \sum_{i\ge1}\binom{|X|+|Y|-2i}{2},\qquad
b(S)=\Bigl\lfloor \tfrac{|X|+|Y|-1}{2}\Bigr\rfloor-|X\cap Y|,
\]
with $b(S)=0$ if $X=Y$. The index $\{b,c\}\subset X\sqcup Y$ is over all the $\binom{2k-1}{2}$ many size 2 subsets of $X\sqcup Y$.

Then the formula in \cite[§13.8]{carter1985} simplifies to
\[
d_{\lambda,\mu}(q)=
\frac{q^{a(S)}\prod_{i=1}^n(q^{2i}-1)}
{2^{\,b(S)}
 \prod_{(b,c)\in \mathrm{hooks}(S)} (q^{c-b}-1)\;
 \prod_{(b,c)\in \mathrm{cohooks}(S)} (q^{c-b}+1)}.
\]

\emph{Example.}  
For $(\lambda,\mu)=((2,2),\varnothing)$ with $n=4$, take $k=2$. Then 
\(X=\{3,2\}\), \(Y=\{0\}\); the hooks are $(0,3),(1,3),(0,2),(1,2)$, the cohooks are $(1,3),(2,3),(1,2)$. Thus $a(S)=2-0$, $b(S)=1$, and
\[
d_{(2,2),\varnothing}(q)=
\frac{q^2(q^2-1)(q^4-1)(q^6-1)(q^8-1)}
{2\,(q-1)(q^2-1)^2(q^3-1)(q+1)^2(q^2+1)}
=\tfrac{1}{2}q^2(1-q+q^2)(1+q^2+q^4+q^6).
\]
At $q=1$ this gives $d(1)=2=\dim S^{(2,2)}$.

\item \emph{Sum of reflections.} 
Let \(z_{\mathrm{short}}=\sum_i \tau_i\) and 
\(z_{\mathrm{long}}=\sum_{i<j}\big((ij)+(ij)\tau_i\tau_j\big)\) be the sums of reflections along short and long roots. Consider the representation $\rho=\rho_{\widetilde S}$.

\emph{Action of $z_{\mathrm{short}}$.}  
We have $z_{\mathrm{short}}\in Z(\C W_n)\cap \C H$, and
\[
\rho_{n_+,n_-}(\tau_i)=
\begin{cases}
  1, & 1\le i\le n_+,\\
 -1, & n_+< i\le n,
\end{cases}
\qquad 
\tau_i \text{ acts trivially on } S^\lambda\boxtimes S^\mu.
\]
Hence 
\(\rho(z_{\mathrm{short}})=n_+-n_-=|\lambda|-|\mu|\).  
Since $z_{\mathrm{short}}$ is central in $\C W_n$, it acts on $S^{\lambda,\mu}$ also by this scalar.

\emph{Action of $z_{\mathrm{long}}$.} 
Write
\(
  z_{\mathrm{long}}
  \;=\;
  \sum_{1\le i<j\le n}\Big((ij)+(ij)\tau_i\tau_j\Big)\;\in Z(\C W_n).
\)
Its restriction (projection) to $\C[H]$ is obtained by keeping only those terms
with permutation part in $S_{n_+}\times S_{n_-}$, i.e.\ with $i,j$ in the same
block. Thus
\[
  z'_{\mathrm{long}}
  \;=\;
  \sum_{1\le i<j\le n_+}\Big((ij)+(ij)\tau_i\tau_j\Big)
  \;+\;
  \sum_{n_+< i<j\le n}\Big((ij)+(ij)\tau_i\tau_j\Big)
  \;\in\;\C[H].
\]
By the induction–restriction scalar equality above, $z_{\mathrm{long}}$ acts on
$S^{\lambda,\mu}$ by the same scalar as $z'_{\mathrm{long}}$ acts on
$\widetilde S^{\lambda,\mu}$. Since 
\(
\rho_{n_+,n_-}(1+\tau_i\tau_j)=2
\)
for $i,j$ within the same block, it follows that
\[
  \rho_{\widetilde S}(z'_{\mathrm{long}})
  \;=\;
  \rho_{S^\lambda}\boxtimes \rho_{S^\mu}\!\left(
    2\sum_{1\le i<j\le n_+}(ij)
    \;+\;
    2\sum_{n_+< i<j\le n}(ij)
  \right),
\]
In $S_m$, the class sum $\sum_{1\leq i< j\leq m} (ij)$ of transpositions acts on $S^\nu$ by the scalar \(\kappa(\nu)=\sum_i \binom{\nu_i}{2}-\sum_j \binom{\nu'_j}{2}\) (see also \S\ref{subsec:typeA}). Therefore \(z'_{\mathrm{long}}\) acts on $S^\lambda\boxtimes S^\mu$ (and hence on $\widetilde S^{\lambda,\mu}$) by \(2\kappa(\lambda)+2\kappa(\mu)\). This is the same scalar for $z_{\mathrm{long}}$ acting on $S^{\lambda,\mu}$. In summary:
\[
   \frac{\chi^{\lambda,\mu}(z_{\mathrm{short}})}{\chi^{\lambda,\mu}(1)}=|\lambda|-|\mu|,
   \qquad
   \frac{\chi^{\lambda,\mu}(z_{\mathrm{long}})}{\chi^{\lambda,\mu}(1)}
   =2(\kappa(\lambda)+\kappa(\mu)).
\]

  \item \emph{Longest element.}
The longest element is $w_0=\tau_1\cdots\tau_n$, which equals $-\Id$ in the standard matrix realization of $W_n$. Being central, $w_0$ acts on each irreducible by a scalar. On $\widetilde S^{\lambda,\mu}=\rho_{n_+,n_-}\otimes(S^\lambda\boxtimes S^\mu)$ we have 
\(\rho_{n_+,n_-}(w_0)=(-1)^{|\mu|}\), while $w_0$ acts trivially on $S^\lambda\boxtimes S^\mu$. Hence the scalar for $w_0$ on both $\widetilde S^{\lambda,\mu}$ and $S^{\lambda,\mu}$ is $(-1)^{|\mu|}$. In particular,
\[
   \chi^{\lambda,\mu}(w_0e_{W_I})
   =(-1)^{|\mu|}\,\chi^{\lambda,\mu}(e_{W_I})
   =(-1)^{|\mu|}\,n_{(\lambda,\mu)}.
\]

\item \emph{Special element $u$.}
When $j=i$,  the type orbit $C$ has length $2$. The two base points are 
\[
v_0=(P_0,\,s_i\to s_{2i}),\qquad v_1=(P_1,\,s_{2i}\to s_i).
\]
Then $w_0:=w_{S\setminus\{s_i\}}$, $w_1:=w_{S\setminus\{s_{2i}\}}$, and $w_0'=w_1'=w_{S\setminus\{s_i,s_{2i}\}}$ are the longest elements in the corresponding parabolics, and the two–step operator is
\[
   D(v_0,v_0,2)=a_{P_0 w_1 P_1}\,a_{P_1 w_0 P_0}=a_{P_0 u P_0},\qquad
   u=w_0' w_1 w_1' w_0.
\]
In the standard realization for $W_n$ one finds explicitly
\[
  u(l)=
  \begin{cases}
    l+i, & 1\le l\le i, \\[4pt]
    -(2i+1-l), & i+1\le l\le 2i, \\[4pt]
    l, & 2i+1\le l\le n.
  \end{cases}
\]

\medskip
More generally, write $n=i+i+k$ and split the coordinates into
\[
I_1=\{1,\dots,i\},\qquad I_2=\{i+1,\dots,2i\},\qquad K=\{2i+1,\dots,2i+k\}.
\]
Then $ue_{W_I}$ is the average of all signed block–bijections sending $I_1\to I_2$, $I_2\to -I_1$, and $K\to\pm K$, altogether $(i!)^2 2^k k!$ elements.

We compute the values of $\chi^{\lambda,\mu}(ue_{W_I})$ for those $\lambda,\mu$ with positive $n_{(\lambda,\mu)}$. By part (3) above, such $\lambda$ has at most three parts and $\mu$ at most two parts.

\emph{Block factorization} Let $|\mu|=n_-=2j$ (if $n_-$ is odd the character value vanishes) and $|\lambda|=n-n_-$. Group the minus coordinates into two blocks $J_3,J_4$ of size $j$, and the plus coordinates into two blocks $J_1,J_2$ of size $i-j$ together with $K_1$ of size $k$. Let $u_2$ swap $J_3\leftrightarrow J_4$ and $u_1$ swap $J_1\leftrightarrow J_2$ while fixing $K_1$ pointwise. Denote by $e_{J_r}$ the averaging idempotent of $S_{J_r}$. Then a direct, but somewhat lengthy, theoretical computation shows that (a detailed derivation and generalization will appear in a future work)
\[
\chi^{\lambda,\mu}(ue_{W_I})=(-1)^{j}\;\chi^{\lambda}\!\bigl(u_1\,e_{J_1}e_{J_2}e_{K_1}\bigr)\cdot
\chi^{\mu}\!\bigl(u_2\,e_{J_3}e_{J_4}\bigr).
\]

\emph{Values.}  
Write $\mu=(\mu_1,\mu_2)$ and $\lambda=(\lambda_1,\lambda_2,\lambda_3)$ in descending order. For the $\mu$ part, it is known \cite[Prop.~8.5.2]{ShenAIM2024} that $\chi^{\mu}(u_2 e_{J_3}e_{J_4})=(-1)^{\mu_1}$. For the $\lambda$ part, write $|\lambda|=2\ell+k$, then
\[
\chi^\lambda\!\bigl(u_1 e_{J_1}e_{J_2}e_{K_1}\bigr)
=\sum_{\lambda',\mu'}c^{\lambda}_{\lambda',\mu'}\,
   \chi^{\lambda'}(u_1 e_{J_1}e_{J_2})\,
   \chi^{\mu'}(e_{K_1})
=\sum_{r=0}^{\ell} (-1)^{r}\, c^{\lambda}_{(2\ell-r,r),\,(k)},
\]
since only $\mu'=(k)$ contributes on $K_1$, and $\chi^{\lambda'}(u_1 e_{J_1}e_{J_2})=(-1)^r$ for $\lambda'=(2\ell-r,r)$ while it vanishes when $\lambda'$ has at least three parts. By Pieri’s rule this sum simplifies to $\tfrac{(-1)^A+(-1)^B}{2}$, where $A=\max\{\lambda_3,\,\lambda_2+\lambda_3-k\}$ and $B=\min\{\lambda_2,\,\ell,\,2\ell-\lambda_2\}$.

\medskip
Combining these gives
\[
\chi^{\lambda,\mu}(ue_{W_I})=
\begin{cases}
(-1)^{\tfrac{|\mu|}{2}+\mu_1}
\,\dfrac{(-1)^{A}+(-1)^{B}}{2}, & |\mu|\ \text{even},\\[10pt]
0,&|\mu|\ \text{odd}.
\end{cases}
\]

In particular, $\chi^{\lambda,\mu}(ue_{W_I})\in \{-1,0,1\}$.
\end{enumerate}

These values supply the input for the invariants $d_\chi$, $f_\chi$, and 
$m_{C,\chi}(\zeta)$ in the next subsection.

\subsubsection{Closed formulas}\label{subsec:typeC-numerics}
We now specialize the general formulas to obtain explicit spectral data for type $\mathbf{C}_n$.  
Fix a type orbit $C$ with initial pair $(i,i{+}j)$, $1\leq i\leq j$, and set
\[
I=S\setminus\{s_i,s_{i+j}\},\qquad P_0=P_I,\qquad 
W_I\cong S_i\times S_{\,j}\times W_{\,k},\quad n=i+j+k.
\]
The orbit length is $c=|C|\in\{2,4\}$, with
\[
c=\begin{cases}
4,& j\neq i,\\
2,& j=i,
\end{cases}
\qquad 
d=\tfrac{2m}{c}\in\{2,4\},\quad 2m=8.
\]

For an irreducible bipartition $(\lambda,\mu)\vdash n$, let $\chi=\chi_q^{\lambda,\mu}$ denote the corresponding representation of $G$, and $\chi_1=\chi_1^{\lambda,\mu}$ the representation of $W_n$. The contribution of $C$ to the zeta factor depends on five invariants \(\ell(w_I),\,d_\chi,\,n_\chi,\,f_\chi,\,m_{C,\chi}(\zeta)\), which we explain in turn.

\paragraph{Explicit invariants}
\emph{Parabolic length.}  
Removing the two nodes leaves components of type $\mathbf{A}_{i-1}$, $\mathbf{A}_{j-1}$, and $\mathbf{C}_{k}$, so
\[
\ell(w_I)=\tfrac{(i-1)i}{2}+\tfrac{(j-1)j}{2}+k^{2}.
\]

\emph{Degrees.}  
The $G$–degree is the generic degree \(d_\chi=d_{\lambda,\mu}(q)\), given by the hook–cohook formula (See \S\ref{subsec:typeC-spectral}(4a)):  
\[
d_{\lambda,\mu}(q)=
\frac{q^{a(S)}\prod_{i=1}^n (q^{2i}-1)}
{2^{b(S)}\prod_{(b,c)\in\mathrm{hooks}(S)} (q^{c-b}-1)\;\prod_{(b,c)\in\mathrm{cohooks}(S)} (q^{c-b}+1)}.
\]

\emph{Multiplicity (\S\ref{subsec:typeC-spectral}(3b)).}  
By the deformation correspondence $1_{W_I}^{W_n}\leftrightarrow \C Ge_{P_0}$, one has
\[
n_\chi=\langle 1_{W_I}^{W_n},\,\chi_1\rangle_{W_n}
=\sum_{a=0}^{i}\ \sum_{b=0}^{j}
K_{\lambda,(a,b,k)}\,K_{\mu,(i-a,\,j-b)}.
\]

\emph{Reflection statistic (\S\ref{subsec:typeC-spectral}(4b)).}  
With \(\kappa(\nu)=\sum_i \binom{\nu_i}{2}-\sum_j \binom{\nu'_j}{2}\), the eigenvalue of the reflection sum is
\[f_\chi =\sum_{t\in \mathcal{R}}
     \Bigl(1+\frac{\chi_1(t)}{\chi_1(1)}\Bigr)=n^2+ \frac{\chi^{\lambda,\mu}(z_{\mathrm{short}})}{\chi^{\lambda,\mu}(1)}+\frac{\chi^{\lambda,\mu}(z_{\mathrm{long}})}{\chi^{\lambda,\mu}(1)}= n^2 + (|\lambda|-|\mu|) + 2(\kappa(\lambda)+\kappa(\mu)).\]

\paragraph{Root splitting}  
The nonzero eigenvalues of the $c$–step operator are of the form 
\(\zeta\,q^{(f_\chi-2\ell(w_I))/d}\) with $\zeta^d=1$, and their multiplicities 
$m_{C,\chi}(\zeta)$ satisfy 
\(\sum_{\zeta^d=1} m_{C,\chi}(\zeta)=n_\chi\). 
By equation~\eqref{eq:mult-identity} in \S\ref{subsec:main-conclusion},
\[
\chi_1(a_{W_0uW_0}^k)=\chi_1(u^ke_{W_0})
=\sum_{\zeta^d=1} m_{C,\chi}(\zeta)\,\zeta^k ,
\]
so the values of $\chi_1(u^ke_{W_0})$ determine the root distribution.

\smallskip
\emph{Generic case ($c=4$, $d=2$).}  
Here $u=w_0$. By \S\ref{subsec:typeC-spectral}(4c), one has 
\(\chi_1(w_0e_{W_0})=(-1)^{|\mu|}n_\chi\). 
Thus the eigenvalues split purely by sign:
\[
m_{C,\chi}(+1)=
\begin{cases}n_\chi,&|\mu|\ \text{even},\\0,&|\mu|\ \text{odd},\end{cases}
\qquad
m_{C,\chi}(-1)=
\begin{cases}0,&|\mu|\ \text{even},\\n_\chi,&|\mu|\ \text{odd}.\end{cases}
\]

\emph{Special case ($c=2$, $d=4$).}  
Now $i=j$ and $u$ is as in \S\ref{subsec:typeC-spectral}(4d), and the eigenvalues lie in 
\(\{1,i,-1,-i\}\). 
We compute their multiplicities using the relations
\[
\sum_{\zeta^4=1} \zeta\,m_{C,\chi}(\zeta)=\chi^{\lambda,\mu}(ue_{W_I}), 
\qquad
\sum_{\zeta^4=1} \zeta^2 m_{C,\chi}(\zeta)=\chi^{\lambda,\mu}(w_0e_{W_I}).
\]
By \S\ref{subsec:typeC-spectral}(4d),
\[
\chi^{\lambda,\mu}(ue_{W_I})=
\begin{cases}
(-1)^{\tfrac{|\mu|}{2}+\mu_1}\,
\dfrac{(-1)^{A}+(-1)^{B}}{2}, & |\mu|\ \text{even},\\[6pt]
0,&|\mu|\ \text{odd},
\end{cases}
\]
with $A=\max\{\lambda_3,\,\lambda_2+\lambda_3-k\}$ and 
$B=\min\{\lambda_2,\,\ell,\,2\ell-\lambda_2\}$ and $\ell=\frac{|\lambda|-k}{2}$. 
Moreover, since $u^2=w_0$, one has 
\(\chi^{\lambda,\mu}(w_0e_{W_I})=(-1)^{|\mu|}n_\chi\).

It follows that:  
If $|\mu|$ is even, then $m_{C,\chi}(i)=m_{C,\chi}(-i)=0$ and 
  \(m_{C,\chi}(\pm1)=(n_\chi\pm\epsilon)/2\), where 
  \(\epsilon=(-1)^{\tfrac{|\mu|}{2}+\mu_1}\tfrac{(-1)^A+(-1)^B}{2}\in\{-1,0,1\}\);  
if $|\mu|$ is odd, then $m_{C,\chi}(1)=m_{C,\chi}(-1)=0$ and 
  \(m_{C,\chi}(i)=m_{C,\chi}(-i)=n_\chi/2\).

\medskip
In summary, the zeta factor contributed by the orbit $C$ is
\[
Z_c(X_2(\mathcal B)|_C,u)
=\prod_{\chi:\,n_\chi\neq 0}\ \prod_{\zeta^{\,d}=1}
\bigl(1-\zeta\,q^{(f_\chi-2\ell(w_I))/d}\,u^{\,c}\bigr)^{-m_{C,\chi}(\zeta)\,d_\chi}.
\]
In the generic case $c=4$ this simplifies to
\[
Z_c(X_2(\mathcal B)|_C,u)
=\prod_{\chi:\,n_\chi\neq 0}
\bigl(1-(-1)^{|\mu|}\,q^{(f_\chi-2\ell(w_I))/2}\,u^{4}\bigr)^{-d_\chi n_\chi}.
\]
In the special case $c=2$, the simplification is given by inserting the above multiplicities 
$m_{C,\chi}(1),m_{C,\chi}(-1),m_{C,\chi}(i),m_{C,\chi}(-i)$.

\paragraph{Examples}
For illustration we record the factorization of the zeta factor in the first nontrivial cases.
For each pair $(n,(i,i{+}j))$ we list the nonzero terms $(\lambda,\mu)\vdash n$
together with the corresponding factor
\[
\prod_{\zeta^{\,d}=1}
\bigl(1-\zeta\,q^{(f_\chi-2\ell(w_I))/d}\,u^{\,c}\bigr)^{[m_{C,\chi}(\zeta)]\,[d_\chi]}
\]
where $c,d$ are as in~\S\ref{subsec:typeC-numerics}.  
The exponents $n_\chi$ and $d_\chi(q)$ are given explicitly; each line shows one irreducible
$(\lambda,\mu)$–contribution to $1/Z_C$.

\medskip
{\small
\noindent\textbf{Case $n=2$, $(i,j)=(1,1)$ (type $\mathbf{C}_2$).}
\[
\begin{array}{|l|l|}
\hline
\textbf{Bipartition $(\lambda,\mu)$} 
& \textbf{Inverse Zeta factor}\\
\hline
((),(1,1))   & (1-q^{0}u^{2})^{[1]\times[q^{4}]}\\[3pt]
((),(2))     & (1+q^{1}u^{2})^{[1]\times[\tfrac{1}{2}q(q^{2}+1)]}\\[3pt]
((1),(1))    & \bigl((1+i q u^{2})(1-i q u^{2})\bigr)^{[1]\times[\tfrac{1}{2}q(q+1)^{2}]}\\[3pt]
((1,1),())   & (1+q^{1}u^{2})^{[1]\times[\tfrac{1}{2}q(q^{2}+1)]}\\[3pt]
((2),())     & (1-q^{2}u^{2})^{[1]\times[1]}\\
\hline
\end{array}
\]
}

\medskip
{\small
\noindent\textbf{Case $n=3$, $(i,j)=(1,2)$ (type $\mathbf{C}_3$).}
\[
\begin{array}{|l|l|}
\hline
\textbf{Bipartition $(\lambda,\mu)$} 
& \textbf{Inverse Zeta factor}\\
\hline
((),(2,1))    & (1+q^{2}u^{4})^{[1]\times[\tfrac{1}{2}(q+1)^{2}q^{4}(q^{2}-q+1)]}\\[3pt]
((),(3))      & (1+q^{5}u^{4})^{[1]\times[\tfrac{1}{2}q(q^{2}-q+1)(q^{2}+1)]}\\[3pt]
((1),(1,1))   & (1-q^{2}u^{4})^{[1]\times[\tfrac{1}{2}q^{4}(q^{2}+1)(q^{2}+q+1)]}\\[3pt]
((1),(2))     & (1-q^{4}u^{4})^{[2]\times[q^{2}(q^{2}-q+1)(q^{2}+q+1)]}\\[3pt]
((1,1),(1))   & (1+q^{3}u^{4})^{[1]\times[q^{3}(q^{2}-q+1)(q^{2}+q+1)]}\\[3pt]
((2),(1))     & (1+q^{5}u^{4})^{[2]\times[\tfrac{1}{2}q(q^{2}+1)(q^{2}+q+1)]}\\[3pt]
((2,1),())    & (1-q^{5}u^{4})^{[1]\times[\tfrac{1}{2}q(q+1)^{2}(q^{2}-q+1)]}\\[3pt]
((3),())      & (1-q^{8}u^{4})^{[1]\times[1]}\\
\hline
\end{array}
\]
}

\medskip
The \textsf{SageMath} scripts for computing the zeta functions and the various invariantsare available at
\href{https://arxiv.org/abs/2405.14395}{arXiv:2405.14395} as the file \texttt{zeta\_buildings\_typeC.ipynb}.

\subsection{Tables of cycles and half-cycle length $m$}\label{sec:4.4}

For each irreducible Coxeter system $(W,S)$, the directed edge types form the set
$S\times S\setminus \Delta(S)$.
We decompose this set into \emph{type orbits} $C$ under the next–type map
(§\ref{sec:partite-decomp}), and we attach to each orbit $C$ an integer $m$
via Luo’s decomposition Theorem~\ref{thm:Luo-decomp}.
The integer $m$ governs the collapse $(T_C)^{2m/c}=D(v_0,v_0,2m)_r$, where $c=|C|$.

\medskip
Some computations were carried out in \textsf{CHEVIE} for \textsf{GAP3};
for exceptional types we follow the simple–reflection labels used in \cite[p.~1592, §85]{GAP344}.
In types $\mathbf{B}/\mathbf{C}/\mathbf{D}$ we place the “branching” at the far end.

\subsubsection*{Part 1: Algorithm}

Let $(W,S)$ be a Coxeter system and write $S=\{s_1,\dots,s_n\}$.
Fix an ordered pair $(t_0,t_1)\in S\times S$ with $t_0\neq t_1$. The next type is recursively defined by
\[
t_{i+1}\;:=\;w_{S-\{t_i\}}\;t_{i-1}\;w_{S-\{t_i\}}\in S\qquad (i\ge 1),
\]
where $w_J$ denotes the longest element of the parabolic subgroup $W_J$.
The resulting orbit is
\[
C=C(t_0,t_1)\;=\;\{(t_i,t_{i+1})\,:\, i\ge 0\},\qquad c=|C|.
\]

Set
\(
w_i:=w_{S-\{t_i\}},w'_i:=w_{S-\{t_i,t_{i+1}\}}
\)
Luo’s theorem implies a factorization of $w'_0w_S$ into the $m$ terms
$w'_{k-1}w_k$ repeated along the cycle, and determines the integer $m$ for which
\[
\ell(w'_0w_S)\;=\;\sum_{k=1}^{m}\ell(w'_{k-1}w_k).
\]
Thus $m$ is obtained by computing the lengths $\ell(w'_{k-1}w_k)$ once on $C$ and summing.

\medskip
The tables in Parts~2–3 record: (i) the orbit decomposition of
$S\times S\setminus\Delta(S)$ under the next–type map, and (ii) the corresponding integer $m$
extracted from the above length identity.

\subsubsection*{Part 2: Classical types: patterns and rank–5 tables}
We use the following Dynkin diagram labeling.

\begin{center}
\begin{tikzpicture}[node distance=1cm, font=\small]
  \node (A1) [circle, draw] {1};
  \node (A2) [circle, draw, right of=A1] {2};
  \node (A3) [right of=A2] {$\cdots$};
  \node (An1) [circle, draw, right of=A3] {n$-$1};
  \node (An) [circle, draw, right of=An1] {n};
  \draw (A1) -- (A2) -- (A3) -- (An1) -- (An);
  \node (LabelA) [below=0.3cm of A3] {Diagram of $\mathbf{A}_n$};

  \node (C1) [circle, draw, right=3.6cm of An] {1};
  \node (C2) [circle, draw, right of=C1] {2};
  \node (C3) [right of=C2] {$\cdots$};
  \node (Cn1) [circle, draw, right of=C3] {n$-$1};
  \node (Cn) [circle, draw, right of=Cn1] {n};
  \draw (C1) -- (C2) -- (C3) -- (Cn1);
  \draw[double] (Cn1) -- (Cn);
  \node (LabelC) [below=0.3cm of C3] {Diagram of $\mathbf{C}_n$};

  \node (D1) [circle, draw, below=3cm of A2] {1};
\node (D2) [circle, draw, right=of D1] {2};
\node (D3) [circle, draw, right=of D2] {3};
\node (Ddots) [right=of D3] {$\cdots$};
\node (Dn2) [circle, draw, right=of Ddots] {$n\!-\!2$};
\node (Dn1) [circle, draw, above right=0.8cm and 1cm of Dn2] {$n\!-\!1$};
\node (Dn)  [circle, draw, below right=0.8cm and 1cm of Dn2] {$n$};

\draw (D1) -- (D2) -- (D3) -- (Ddots) -- (Dn2);
\draw (Dn2) -- (Dn1);
\draw (Dn2) -- (Dn);

\node (LabelD) [below=1cm of Dn2] {Diagram of $\mathbf{D}_n$};

\end{tikzpicture}
\end{center}

\paragraph{Type $\mathbf{A}_{n-1}$}
Here $W=S_n$.
Write $n=i+j+k$ with $i,j,k\ge 1$ and start at the pair $(t_0,t_1)=(s_i,s_{i+j})$.
Then the orbit is
\[
C=\{(i,i{+}j),\,(i{+}j,j),\,(j,j{+}k),\,(j{+}k,k),\,(k,k{+}i),\,(k{+}i,i)\},
\]
so $c=6$ generically and $c=2$ in the balanced case $i=j=k$.
A direct length computation gives
\[
\ell(w'_0w_1)=ij,\quad \ell(w'_1w_2)=ik,\quad \ell(w'_2w_3)=jk,\quad
\ell(w'_0w_S)=\tbinom{n}{2}- \tbinom{i}{2}-\tbinom{j}{2}-\tbinom{k}{2}=ij+ik+jk,
\]
hence $\ell(w'_0w_S)=\sum_{r=0}^2 \ell(w'_r w_{r+1})$ and therefore
\(
{\,m=3\ \text{ in type } A_{n-1}.}
\)

\paragraph{Type $\mathbf{C}_n$}
Here $W=C_2^{\,n}\rtimes S_n$.
Writing $n=i+j+k$ with $i,j\ge 1$ and $k\ge 0$, the orbit through $(s_i,s_{i+j})$ is 
\[
C=\{(i,i{+}j),\,(i{+}j,j),\,(j,i{+}j),\,(i{+}j,i)\},
\]
so $c=4$ generically and $c=2$ when $i=j$.

A parabolic–length check along the four segments of the cycle shows that the sum of
$\ell(w'_{r-1}w_r)$ over one revolution equals $\ell(w'_0w_S)$, hence
\(
{\,m=4\ \text{ in type } C_n.}
\)

\paragraph{Type $\mathbf{D}_n$}
In $\mathbf{D}_n$ several orbit sizes occur ($c\in\{2,3,4,6,8\}$), depending on whether the pair $(t_0,t_{1})$
meets the branching nodes $\{s_{n-1},s_n\}$, whether the label for $t_1$ is double that of $t_0$, and on whether .
A length computation as above yields
\[
{\,m=4\ \text{ generically in type } D_n,\quad
m=3\ \text{ if both } t_0,t_1\in\{s_1,s_{n-1},s_n\}\, .}
\]

\medskip
For concreteness we record rank–5 cycles and $m$ for $\mathbf{A}_5$, $\mathbf{C}_5$, and $\mathbf{D}_5$.

\begin{table}[H]
\centering
\caption{Table for $\mathbf{A}_5$}
\begin{tabular}{|m{8cm}|c|}
\hline
\textbf{Cycles $ s_0, s_1, s_2, \ldots, s_c$} & \textbf{$ m $} \\ \hline
$ 1 \rightarrow 2 \rightarrow 1 \rightarrow 5 \rightarrow 4 \rightarrow 5 \rightarrow 1 $ & 3 \\ \hline
$ 1 \rightarrow 3 \rightarrow 2 \rightarrow 5 \rightarrow 3 \rightarrow 4 \rightarrow 1 $ & 3 \\ \hline
$ 1 \rightarrow 4 \rightarrow 3 \rightarrow 5 \rightarrow 2 \rightarrow 3 \rightarrow 1 $ & 3 \\ \hline
$ 2 \rightarrow 4 \rightarrow 2 $ & 3 \\ \hline
\end{tabular}
\end{table}

\begin{table}[H]
\centering
\caption{Table for $\mathbf{C}_5$ (labels in $\mathbf{C}$-orientation)}
\begin{tabular}{|m{8cm}|c|}
\hline
\textbf{Cycles $ s_0, s_1, s_2, \ldots, s_c$} & \textbf{$ m $} \\ \hline
$ 1 \rightarrow 2 \rightarrow 1 $ & 4 \\ \hline
$ 1 \rightarrow 3 \rightarrow 2 \rightarrow 3 \rightarrow 1 $ & 4 \\ \hline
$ 1 \rightarrow 4 \rightarrow 3 \rightarrow 4 \rightarrow 1 $ & 4 \\ \hline
$ 1 \rightarrow 5 \rightarrow 4 \rightarrow 5 \rightarrow 1 $ & 4 \\ \hline
$ 2 \rightarrow 4 \rightarrow 2 $ & 4 \\ \hline
$ 2 \rightarrow 5 \rightarrow 3 \rightarrow 5 \rightarrow 2 $ & 4 \\ \hline
\end{tabular}
\end{table}

\begin{table}[H]
\centering
\caption{Table for $\mathbf{D}_5$}
\begin{tabular}{|m{9cm}|c|}
\hline
\textbf{Cycles $ s_0, s_1, s_2, \ldots, s_c$} & \textbf{$ m $} \\ \hline
$ 1 \to 2 \to 1 $ & 4 \\ \hline
$ 1 \to 3 \to 2 \to 3 \to 1 $ & 4 \\ \hline
$ 1 \to 4 \to 5 \to 1 \to 5 \to 4 \to 1 $ & 3 \\ \hline
$ 2 \to 4 \to 3 \to 4 \to 2 \to 5 \to 3 \to 5 \to 2 $ & 4 \\ \hline
\end{tabular}
\end{table}

\subsubsection*{Part 3: Exceptional types: Complete lists}

For each exceptional type $\mathbf{G}_2,\mathbf{F}_4,\mathbf{E}_6,\mathbf{E}_7,\mathbf{E}_8$, we list the type orbits
and the associated $m$ (labels as in \textsf{GAP3}/\textsf{CHEVIE}). 
The verification uses the same length–sum criterion
$\ell(w'_0w_S)=\sum \ell(w'_{k-1}w_k)$ along the cycle.

\begin{center}
\begin{tikzpicture}[node distance=1cm, font=\small]
    \node (E6-1) [circle, draw, ] {1};
    \node (E6-2) [circle, draw, right of=E6-1] {3};
    \node (E6-3) [circle, draw, right of=E6-2] {4};
    \node (E6-4) [circle, draw, right of=E6-3] {5};
    \node (E6-5) [circle, draw, right of=E6-4] {6};
    \node (E6-6) [circle, draw, below=1cm of E6-3]{2};
    \draw (E6-1) -- (E6-2) -- (E6-3) -- (E6-4) -- (E6-5);
    \draw (E6-3) -- (E6-6);

    \node (E7-1) [circle, draw, right=1cm of E6-5] {1};
    \node (E7-2) [circle, draw, right of=E7-1] {3};
    \node (E7-3) [circle, draw, right of=E7-2] {4};
    \node (E7-4) [circle, draw, right of=E7-3] {5};
    \node (E7-5) [circle, draw, right of=E7-4] {6};
    \node (E7-6) [circle, draw, right of=E7-5] {7};
    \node (E7-7) [circle, draw, below=1cm of E7-3] {2};
    \draw (E7-1) -- (E7-2) -- (E7-3) -- (E7-4) -- (E7-5) -- (E7-6);
    \draw (E7-3) -- (E7-7);

    \node (E8-1) [circle, draw, below=3cm of E6-1] {1};
    \node (E8-2) [circle, draw, right of=E8-1] {3};
    \node (E8-3) [circle, draw, right of=E8-2] {4};
    \node (E8-4) [circle, draw, right of=E8-3] {5};
    \node (E8-5) [circle, draw, right of=E8-4] {6};
    \node (E8-6) [circle, draw, right of=E8-5] {7};
    \node (E8-7) [circle, draw, right of=E8-6] {8};
    \node (E8-8) [circle, draw, below=1cm of E8-3] {2};
    \draw (E8-1) -- (E8-2) -- (E8-3) -- (E8-4) -- (E8-5) -- (E8-6) -- (E8-7);
    \draw (E8-3) -- (E8-8);
    \node (LabelE) [below =1cm of E8-6]  {Diagram of $\mathbf{E}_6,\mathbf{E}_7,\mathbf{E}_8$};
\end{tikzpicture}
\end{center}

\begin{table}[H]
\centering
\caption{Table for $\mathbf{G}_2$}
\begin{tabular}{|m{4cm}|c|}
\hline
\textbf{Cycles $ s_0, s_1, s_2, \ldots, s_c$} & \textbf{$ m $} \\ \hline
$ 1 \rightarrow 2 \rightarrow 1 $ & 6 \\ \hline
\end{tabular}
\end{table}

\begin{table}[H]
\centering
\caption{Table for $\mathbf{F}_4$}
\begin{tabular}{|m{9cm}|c|}
\hline
\textbf{Cycles $ s_0, s_1, s_2, \ldots, s_c$} & \textbf{$ m $} \\ \hline
$ 1 \rightarrow 2 \rightarrow 1 $ & 6 \\ \hline
$ 1 \rightarrow 3 \rightarrow 2 \rightarrow 4 \rightarrow 2 \rightarrow 3 \rightarrow 1 $ & 6 \\ \hline
$ 1 \rightarrow 4 \rightarrow 1 $ & 4 \\ \hline
$ 3 \rightarrow 4 \rightarrow 3 $ & 6 \\ \hline
\end{tabular}
\end{table}

\begin{table}[H]
\centering
\caption{Table for $\mathbf{E}_6$}
\begin{tabular}{|m{12cm}|c|}
\hline
\textbf{Cycles $ s_0, s_1, s_2, \ldots, s_c$} & \textbf{$ m $} \\ \hline
$ 1 \rightarrow 2 \rightarrow 6 \rightarrow 5 \rightarrow 6 \rightarrow 2 \rightarrow 1 \rightarrow 3 \rightarrow 1 $ & 4 \\ \hline
$ 1 \rightarrow 4 \rightarrow 3 \rightarrow 5 \rightarrow 4 \rightarrow 6 \rightarrow 4 \rightarrow 5 \rightarrow 3 \rightarrow 4 \rightarrow 1 $ & 5 \\ \hline
$ 1 \rightarrow 5 \rightarrow 2 \rightarrow 3 \rightarrow 6 \rightarrow 3 \rightarrow 2 \rightarrow 5 \rightarrow 1 $ & 4 \\ \hline
$ 1 \rightarrow 6 \rightarrow 1 $ & 3 \\ \hline
$ 2 \rightarrow 4 \rightarrow 2 $ & 6 \\ \hline
\end{tabular}
\end{table}

\begin{table}[H]
\centering
\caption{Table for $\mathbf{E}_7$}
\begin{tabular}{|m{12cm}|c|}
\hline
\textbf{Cycles $ s_0, s_1, s_2, \ldots, s_c$} & \textbf{$ m $} \\ \hline
$ 1 \rightarrow 2 \rightarrow 7 \rightarrow 2 \rightarrow 1 $ & 4 \\ \hline
$ 1 \rightarrow 3 \rightarrow 1 $ & 6 \\ \hline
$ 1 \rightarrow 4 \rightarrow 3 \rightarrow 6 \rightarrow 3 \rightarrow 4 \rightarrow 1 $ & 6 \\ \hline
$ 1 \rightarrow 5 \rightarrow 2 \rightarrow 4 \rightarrow 2 \rightarrow 5 \rightarrow 1 $ & 6 \\ \hline
$ 1 \rightarrow 6 \rightarrow 1 $ & 4 \\ \hline
$ 1 \rightarrow 7 \rightarrow 6 \rightarrow 7 \rightarrow 1 $ & 4 \\ \hline
$ 2 \rightarrow 3 \rightarrow 7 \rightarrow 5 \rightarrow 6 \rightarrow 2 $ & 5 \\ \hline
$ 2 \rightarrow 6 \rightarrow 5 \rightarrow 7 \rightarrow 3 \rightarrow 2 $ & 5 \\ \hline
$ 3 \rightarrow 5 \rightarrow 4 \rightarrow 7 \rightarrow 4 \rightarrow 5 \rightarrow 3 $ & 6 \\ \hline
$ 4 \rightarrow 6 \rightarrow 4 $ & 6 \\ \hline
\end{tabular}
\end{table}

\begin{table}[H]
\centering
\caption{Table for $\mathbf{E}_8$}
\begin{tabular}{|m{12cm}|c|}
\hline
\textbf{Cycles $ s_0, s_1, s_2, \ldots, s_c$} & \textbf{$ m $} \\ \hline
$ 1 \rightarrow 2 \rightarrow 8 \rightarrow 2 \rightarrow 1 \rightarrow 3 \rightarrow 1 $ & 6 \\ \hline
$ 1 \rightarrow 4 \rightarrow 3 \rightarrow 7 \rightarrow 5 \rightarrow 7 \rightarrow 3 \rightarrow 4 \rightarrow 1 $ & 8 \\ \hline
$ 1 \rightarrow 5 \rightarrow 2 \rightarrow 5 \rightarrow 1 $ & 8 \\ \hline
$ 1 \rightarrow 6 \rightarrow 1 $ & 6 \\ \hline
$ 1 \rightarrow 7 \rightarrow 6 \rightarrow 8 \rightarrow 6 \rightarrow 7 \rightarrow 1 $ & 6 \\ \hline
$ 1 \rightarrow 8 \rightarrow 1 $ & 4 \\ \hline
$ 2 \rightarrow 3 \rightarrow 8 \rightarrow 3 \rightarrow 2 \rightarrow 7 \rightarrow 2 $ & 6 \\ \hline
$ 2 \rightarrow 4 \rightarrow 2 \rightarrow 6 \rightarrow 5 \rightarrow 8 \rightarrow 5 \rightarrow 6 \rightarrow 2 $ & 8 \\ \hline
$ 3 \rightarrow 5 \rightarrow 4 \rightarrow 8 \rightarrow 4 \rightarrow 5 \rightarrow 3 \rightarrow 6 \rightarrow 3 $ & 8 \\ \hline
$ 4 \rightarrow 6 \rightarrow 4 \rightarrow 7 \rightarrow 4 $ & 8 \\ \hline
$ 7 \rightarrow 8 \rightarrow 7 $ & 6 \\ \hline
\end{tabular}
\end{table}

\section*{Further Directions}
\addcontentsline{toc}{section}{Further Directions}
The results of this paper suggest some natural extensions, both geometric and representation–theoretic.

\paragraph{(1) Higher zeta functions}  
The edge zeta function considered here encodes geodesics in the $1$–skeleton of the building.  
A next step is to define and study higher–dimensional zeta functions that count closed geodesics in higher skeleton of the building.  
At the top dimension, one obtains the \emph{gallery zeta function}, enumerating closed chamber galleries and generalizing the edge zeta function considered here.  
Such higher zeta functions should encode finer representation–theoretic and geometric information, potentially connecting to the cohomology of the building and to higher–rank analogues of Ihara’s determinant formulas.

\paragraph{(2) Regular elements and character values}  
In Section~\ref{subsec:typeC-spectral}, the quantities $\chi^{\lambda,\mu}(u e_{W_I})$ in type~$\mathbf{C}$ (and their analogues $\chi^{\lambda}(u e_{S_\mu})$ in~\cite[\S8.3]{ShenAIM2024}) take values in $\{0,\pm1\}$.  
This pattern resembles the character values of Coxeter elements or some regular elements in Weyl groups~\cite{springer1974}.

It would be interesting to formalize this connection by defining \emph{block Coxeter elements} or more generally \emph{generalized regular elements} in parabolic subgroups, and to study their character values systematically.  
Such a theory might provide a new representation–theoretic framework for understanding the sign and vanishing patterns observed in these zeta–factor computations.

\begin{acknowledgements}
  This research was funded by the Deutsche Forschungsgemeinschaft (DFG,
  German Research Foundation) -- Project-ID 491392403 -- TRR 358.
\end{acknowledgements}